\documentclass[14pt,twoside,a4paper]{article}
\usepackage{mathrsfs}
\usepackage{mathrsfs}
\usepackage{mathrsfs}
\usepackage{amsfonts}
\usepackage{mathrsfs}
\usepackage{amscd}
\usepackage{amsmath}
\usepackage{latexsym}
\usepackage{amsfonts}
\usepackage{amssymb}
\usepackage{amsthm}
\usepackage{graphicx}
\usepackage{color}
\usepackage{cite}
\usepackage{mathrsfs}
\usepackage{mathrsfs}
\usepackage{mathrsfs}
\usepackage{mathrsfs}
\usepackage{mathrsfs}
\usepackage{mathrsfs}
\usepackage{mathrsfs}
\usepackage{stmaryrd}
\usepackage{amsfonts}
\usepackage{amsmath,amssymb}
\usepackage{mathrsfs}
\usepackage{hyperref}
\renewcommand\baselinestretch{1.2}

\newtheorem{proposition}{Proposition}[section]

\newtheorem{lemma}{Lemma}[section]
\newtheorem{theorem}{Theorem}[section]

\numberwithin{equation}{section}\allowdisplaybreaks

\def\le{\leqslant}
\def\ge{\geqslant}
\def\leq{\leqslant}
\def\geq{\geqslant}
\def\no{\nonumber}
\def\pend{\hfill $\Box$}
\def\Real{{\mathbb{R}}}
\def\NN{{\mathbb{N}}}
\def\FF{{\mathscr{F}^{-1}}}
\def\F{{\mathscr{F}}}
\newcommand{\norm}[1]{\left\|#1\right\|}
\newcommand{\abs}[1]{\left|#1\right|}
\newcommand{\set}[1]{\left\{#1\right\}}
\newcommand{\wuhao}{\fontsize{9pt}{\baselineskip}\selectfont}
\renewcommand{\baselinestretch}{1.2}
\begin{document}

\title{
	Scattering for the magnetic  Zakharov  system in 3 dimensions  }
\author
         {
            {Xiaohong Wang},\ \ \     Lijia Han\footnote{ E--mail address:
 hljmath@ncepu.edu.cn (L. Han)}\\
   {\small Department of Mathematics and Physics, North China Electric Power University, Beijing 102206, China}\\
         \date{}
         }
\maketitle

\begin{minipage}{13.5cm}
\footnotesize \bf Abstract. \rm  \quad We consider the global existence and scattering for solutions of magnetic Zakharov system in  three-dimensional space. When the  initial datas are small,  we prove the existence of smooth global solutions and scattering results, by combining the space--time resonance method, weighted Sobolev space and dispersive estimates.  Moreover,  the decay rates for the solutions are also obtained.

\vspace{10pt}



\end{minipage}

%
%
%
%
%

\section{Introduction}

\setcounter{section}{1}\setcounter{equation}{0} In this paper, we
study the Cauchy problem for the magnetic Zakharov system (MZ)
given by
\begin{equation}\label{1.1}
	\left\{\!\!
	\begin{array}{lc}
		i\mathcal{E}_{t}+\Delta \mathcal{E}-n\mathcal{E}+i \mathcal{E}\times \mathcal{B}=0, &\\
		\alpha^{-2}n_{tt}-\Delta n=\Delta |\mathcal{E}|^{2},&\\
		\beta^{-2}\mathcal{B}_{tt}+(\Delta
		^{2}-\Delta)\mathcal{B}=i\Delta\nabla\times(\nabla\times(\mathcal{E}\times
		\overline{\mathcal{E}})).\\
	\end{array}
	\right.
\end{equation}
with initial datas
\begin{align}
 \mathcal{E}(0, x)=\mathcal{E}_0, \quad (n(0, x),n_t(0, x))=(n_0, n_1), \quad (\mathcal{B}(0, x),
\mathcal{B}_t(0, x))= (\mathcal{B}_0, \mathcal{B}_1),\nonumber
\end{align}
 where, $\alpha>0$ denotes the ion sound speed, and $\beta>0$ denotes the speed of the
 electron. The function
$\mathcal{E}=\mathcal{E}(t,x):\mathbb{R}^{1+3}\rightarrow \mathbb{C}^{3}$ is the
slowly varying amplitude of the high-frequency electric field, the
function $n=n(t,x):\mathbb{R}^{1+3}\rightarrow \mathbb{R}$ is the
fluctuation of the ion-density from its equilibrium, and
$\mathcal{B}:\mathbb{R}^{1+3}\rightarrow \mathbb{R}^{3}$ is the
self-generated magnetic field. $\bar{\mathcal{E}}$ denotes the
conjugate complex of $\mathcal{E}=\mathcal{E}(t,x)$, and the notation $"\times"$ is
the cross product of $\mathbb{R}^{3}$ or $\mathbb{C}^{3}$ valued
vectors. In the case $ \nabla\times\mathcal{E}=0$, it is easy to see $\nabla\times(\nabla\times(\mathcal{E}\times
\overline{\mathcal{E}}))= -\Delta (\mathcal{E}\times
\overline{\mathcal{E}})$.

 Magnetic Zakharov equation is a system coupled by Zakharov equation and the equation describing the self generated magnetic field. The self generated magnetic field
was first found in 1971 in studying the laser plasma (see \cite{4}). Magnetic Zakharov system was first proposed in  \cite{H} when studying the self generated magnetic effect in plasma, which described the gravity force and magnetic field generation effect caused by the nonlinear interaction between plasma wave and particles. In recent
years, with the development of laser plasma, astronomy plasma and the strong turbulence theory related
to them, the self-generated magnetic effect and the equations which describe it received more and more
attention from researchers.

Magnetic Zakharov system has great relationship with classical Zakharov system. When $\beta \rightarrow \infty$, solution of equation \eqref{1.1} will converge to that of generalized Zakharov  equation:
\begin{equation}\label{1.2}
	\left\{\!\!
	\begin{array}{lc}
		i\mathcal{E}_{t}+\Delta \mathcal{E}-n\mathcal{E}+\mathcal{E}\times(I-\Delta)^{-1} \nabla\times(\nabla\times(\mathcal{E}\times
		\overline{\mathcal{E}}))=0, &\\
		\alpha^{-2}n_{tt}-\Delta n=\Delta |\mathcal{E}|^{2}.\\
	\end{array}
	\right.
\end{equation}
When $(\alpha, \beta) \rightarrow \infty$ together, solutions of equation \eqref{1.1} will further converge to that of the nonlinear Schr$\ddot{o}$dinger equation with magnetic effects (MNLS) \eqref{1.3}:
\begin{equation}\label{1.3}
		i\mathcal{E}_{t}+\Delta \mathcal{E}+|\mathcal{E}|^2\mathcal{E}+\mathcal{E}\times(I-\Delta)^{-1} \nabla\times(\nabla\times(\mathcal{E}\times
		\overline{\mathcal{E}}))=0.
\end{equation}

During the past few years,  systems \eqref{1.1}--\eqref{1.3}  have been widely studied by many researchers in mathematics. Focus on  magnetic Zakharov system \eqref{1.1}, in \cite{ZGG},  the local well-posedness results  were obtained. In \cite{BJ}, the low regularity solutions were obtained. The  blow-up results about \eqref{1.1} were considered in \cite{ZBD,ZBLJ,ZYT3}. For the limit behavior results of \eqref{1.1},   \cite{HZGG} proved that the solutions of magnetic Zakharov system converged to those of generalized Zakharov system in Sobolev space $H^s$, $s > 3/2$, when parameter $\beta \rightarrow \infty$; when
$(\alpha, \beta)\rightarrow \infty$ together, they further proved that the solutions of magnetic Zakharov system converged to those of Schr$\ddot{o}$dinger
equation with magnetic effect in  $H^s$, $s > 3/2$.   For other results about magnetic Zakharov equation, we can refer to \cite{ZHG,HZGG2,KHP,WG,WG2,YZZ} and the references therein.

There are many results about the classical Zakharov system,
\begin{equation}\label{1.2-1}
	\left\{\!\!
	\begin{array}{lc}
		i\mathcal{E}_{t}+\Delta \mathcal{E}-n\mathcal{E}=0, &\\
		\alpha^{-2}n_{tt}-\Delta n=\Delta |\mathcal{E}|^{2}.\\
	\end{array}
	\right.
\end{equation}
The wellposedness, low regularity, blow up and limit behavior results were systematically studied in  \cite{JJ,JTG,DHS,M,MN2,OT1,OT2} and the references therein.  For the scattering results, in \cite{ZK}, the authors proved small energy scattering for the 3D Zakharov system under radial symmetric assumption. In order to delete the radial symmetric condition, another way is the application of
 "space--time resonance" method and weighted Sobolev space.
"Space--time resonance" method was developed in  \cite{PNJ1,PNJ2} and further applied to many other dispersive PDEs \cite{PN,G,PNJ3,PNJ4,KP,ZFJ}, such as Klein--Gorden equation \cite{G}, Zakharov equation \cite{ZFJ}, Euler--Maxwell equation \cite{PN,GAB} and so on. In \cite{ZFJ}, the authors utilized  the space--time resonance method  as  general framework  to obtain global existence and scattering results for small localized solutions  in 3 space dimension for Zakharov equation. \cite{GBK2} even showed us a more general frame for this kind of method.

As far as we know, there are no scattering results about the magnetic Zakharov system.
Inspired by \cite{ZFJ}, the  main purpose of this paper is to prove global existence and scattering for solutions of magnetic Zakharov system
in 3 space dimension.  Let
$$
N_{\pm}:= i|\alpha\nabla|^{-1}n_{t}\pm n,\ \
M_{\pm}:= i\beta^{-1}(\Delta^{2}-\Delta)^{-1/2}\mathcal{B}_{t}\pm \mathcal{B}.
$$
 Then, system \eqref{1.1} is reduced to
\begin{equation}\label{3.1}
	\begin{split}
		\left\{\!\!
		\begin{array}{lc}
			i\mathcal{E}_{t}+ \Delta \mathcal{E}=\frac{1}{2}(N_{+}-N_{-})\mathcal{E}+\frac{i}{2} \mathcal{E}\times(M_{+}-M_{-}),
&\\ i\dot{N_{\pm}}\mp |\alpha\nabla|N_{\pm}=|\alpha\nabla| |\mathcal{E}|^{2},&\\
			i\dot{M_{\pm}}\mp \beta(\Delta^{2}-\Delta)^{1/2}M_{\pm}=-i\beta(-\Delta)^{1/2}(I-\Delta)^{-1/2}\Delta(\mathcal{E}\times
		\overline{\mathcal{E}}).\\
		\end{array}
		\right.
	\end{split}
\end{equation}
Next, we will consider the system \eqref{3.1}. We mainly apply  the "space--time resonance" method, combining the weighted Sobolev space
and dispersive estimates, see also in \cite{ZFJ}. In magnetic Zakharov system, some new "space--time resonance" will appear, see \eqref{4.11}--\eqref{4.14}.  For the case $\beta\neq1$, we use frequency decomposition and other techniques to treat it, see section \ref{sec1.1}.  However, when $\beta=1$, the phase $\phi(\xi,\eta)$  don{'}t have good lower bound, so  we need to find other methods to study. Since $\alpha$ does{'}t effect the estimate, so we could assume $\alpha=1$ without loss of generility. We also show that the decay rate for magnetic component $\mathcal{B}$ is  $t^{-1}$, the same as $ n$ which  decays at a rate of $t^{-1}$, and the Schr$\ddot{o}$dinger component $\mathcal{E}$  decays almost at a rate of $t^{-7/6}$.

Notations:\\
$\bullet\quad$Japanese brackets: ${\left\langle x \right\rangle}=\sqrt{1+x^2}$.\\
$\bullet\quad$If $f$ is a function over $R^{3}$ then its Fourier transform, denoted by $\widehat f$ or $\mathcal{F}f$, is given by
\begin{align*}
	\widehat{f}(\xi)=\mathcal{F}f(\xi)=\frac{1}{2\pi}\int_{R^{3}} e^{-ix\xi}f(x)dx.
\end{align*}
$\bullet\quad$Inhomogeneous Sobolev spaces $H^{N}$ is given by the norm $\|f\|_{H^{N}}=\|(1+\wedge)^{N} f\|_{L^{2}}$.\\
Inhomogeneous Sobolev spaces $W^{s,p}$ is given by the norm  $\|f\|_{W^{s,p}}=\|(1+\wedge)^{s} f\|_{L^{p}}$.\\
 We use $R$ to denote indistinctly any one of the components of the vector of Riesz transforms $R=\frac{\nabla}{\wedge}$, where $\wedge$: =$|\nabla|=\sqrt{-\Delta}$.\\
 $\bullet\quad$ Littlewood--Paley theory and Bernstein{'}s inequality.( \cite{HZGG})\\\label{1.11}
Let $\psi\in
C_0^\infty(\mathbb{R}^{d})$ be an even, radial function with the
properties $0\leq \psi\leq 1$, $\psi(\xi)=1$ for $|\xi|\leq 1$ and
$\psi(\xi)=0$ for $|\xi|\geq 2$. Then we write $\phi_0 =\psi$ and
$\phi_j(\xi)=\psi(\frac{\xi}{2^j})-\psi(\frac{\xi}{2^{j-1}})$ for $j
\geq 1$. In this way we have $1=\sum_{j\geq 0}\phi_j$. Define dyadic
frequency localization operators $P_j$ by
$$ \mathscr{F}_x(P_j f)(\xi)=
\phi_j(|\xi|)\mathscr{F}_xf(\xi),
$$
so, all these operators are bounded on $L^{p}$ spaces:
	\begin{align*}
if\quad 1< p< \infty,\quad\|P_{j}f\|_{L^{p}}\lesssim\|f\|_{L^{P}},\quad\|P_{< j}f\|_{L^{p}}\lesssim\|f\|_{L^{P}}.
\end{align*}
Thus, Bernstein{'}s inequality:
	\begin{align*}
	if\quad 1\le p\le q \le \infty,\quad\|P_{j}f\|_{L^{q}}\le2^{3j(\frac{1}{q}-\frac{1}{p})}\|P_{j}f\|_{L^{p}},\quad\|P_{< j}f\|_{L^{q}}\le2^{3j(\frac{1}{q}-\frac{1}{p})}\|P_{< j}f\|_{L^{p}}.
\end{align*}
 $\bullet\quad$The Besov space $B^{s}_{p,r}(\mathbb{R}^{d})$, $1\leq p,r <\infty$, denotes the completion of the Schwartz function
space $\mathscr{S}(\mathbb{R}^{d})$ with
 respect to the norm
$$
\|f\|_{B^{s}_{p,r}}:=\Big(\sum_{j\geq 0}2^{j sr}\|P_j
f\|_{L^{p}}^{r}\Big)^{1/r}.
$$
 $\bullet\quad$${A}\lesssim{B}$ if ${A}\le{CB}$ for some implicit constant $C$, the value of $C$ may change from line to line.\\
\subsection{Main results}
The main results in this paper are the  following theorem.
\begin{theorem}\label{thm1.1}
When $\beta\neq1$, suppose that the initial datas $\mathcal{E}_{0}$, $n_{0}$, $n_{1}$, $\mathcal{B}_{0}$, $\mathcal{B}_{1}$ satisfy
\begin{align}
		&\quad\quad\quad\quad\quad\quad\quad\quad\quad\quad\quad\quad\quad\quad\|\mathcal{E}_{0}\|_{H^{N+1}}+\|{\left\langle x \right\rangle}^{2}\mathcal{E}_{0}\|_{L^{2}}\le \epsilon_{0},\\
		&\|\wedge n_{0}\|_{H^{N-1}}+\|n_{1}\|_{H^{N-1}}+\|\wedge^{2} n_{0}\|_{B_{1,1}^{0}}+\|\wedge n_{1}\|_{B_{1,1}^{0}}+\|{\left\langle x \right\rangle}n_{0}\|_{H^{1}}+\|{\left\langle x \right\rangle}^{2}n_{1}\|_{H^{1}}\le \epsilon_{0},\\
		&\| \mathcal{B}_{0}\|_{H^{N-1}}+\|\mathcal{B}_{1}\|_{H^{N-3}}+\|\mathcal{B}_{1}\|_{L^{1}}+\|{\left\langle x \right\rangle} \mathcal{B}_{0}\|_{H^{N-1}}+\|{\left\langle x \right\rangle} \mathcal{B}_{1}\|_{H^{N-1}}+\|{\left\langle x \right\rangle}^{2} \mathcal{B}_{0}\|_{H^{1}}+\|{\left\langle x \right\rangle}^{2} \mathcal{B}_{1}\|_{H^{1}}\le \epsilon_{0},\label{1.7}
\end{align}
for some small $\epsilon_{0}$ and some large integer $N$.  Then the Cauchy problem for the magnetic Zakharov
system \eqref{1.1} admits a unique global solution such that
\begin{align*}
	&\|\mathcal{E}(t)\|_{L^{\infty}}\lesssim\frac{\epsilon_{0}}{t^{7/6-}},\quad\quad\|\mathcal{E}(t)\|_{H^{N+1}}\lesssim t^{\delta},	\quad\quad\|x\mathcal{E}(t)\|_{L^{2}}\lesssim t^{\delta},\quad\quad\|x^2\mathcal{E}(t)\|_{L^{2}}\lesssim t^{1-2\alpha-\delta},\\
&\|n(t)\|_{L^{\infty}}\lesssim\frac{\epsilon_{0}}{t} ,\quad\quad\quad\|n(t)\|_{H^{N}}\lesssim \epsilon_{0} ,\quad\quad\quad\|xn(t)\|_{L^{2}}\lesssim\epsilon_{0} ,\quad\quad\|\wedge x^2n(t)\|_{L^{2}}\lesssim t^{1-3\alpha},\\
&\|\mathcal{B}(t)\|_{L^{\infty}}\lesssim\frac{\epsilon_{0}}{t},\quad\quad\quad\|\mathcal{B}(t)\|_{H^{N-1}}\lesssim \epsilon_{0},\quad\quad\|x\mathcal{B}(t)\|_{L^{2}}\lesssim\epsilon_{0},\quad\quad\|\wedge x^2\mathcal{B}(t)\|_{L^{2}}\lesssim t^{1-3\alpha}.
\end{align*}
where the parameters could be chosen as
\begin{align}\label{1.9}
\alpha=\frac{1}{6}-2\delta, \quad\beta=1-3\alpha, \quad\frac{5}{N}\lesssim\delta, \quad\forall\delta \ll 1.
\end{align}

Furthermore, let $(\mathcal{E},n,\mathcal{B})(t)$ are the solution of \eqref{1.1}, then the solution scatters as $t\rightarrow \pm\infty$.   More precisely, for the solutions $(\mathcal{E}_{+}, n_{\pm}, \mathcal{B}_{\pm})$, where $\mathcal{E}_{+}$  satisfies $i\dot{\mathcal{E}_{+}}+ \Delta \mathcal{E}_{+}=0$, $n_{\pm}$ satisfy $ i\dot{n_{\pm}}\mp |\alpha\nabla|n_{\pm}=0$, and $\mathcal{B}_{\pm}$ satisfy $
			i\dot{\mathcal{B}_{\pm}}\mp \beta(\Delta^{2}-\Delta)^{1/2}\mathcal{B}_{\pm}=0$, then we have
\begin{align}
&	\|\mathcal{E}(t)-\mathcal{E}_{+}(t)\|_{H^{N+1}}+
\|n(t)- n_{\pm}(t)\|_{H^{N}}+
\|\mathcal{B}(t)-\mathcal{B}_{\pm}(t)\|_{H^{N-1}}\rightarrow0,\quad as \quad t\rightarrow \pm\infty.
\end{align}
\end{theorem}

%
%
%
%
%

\section{Preliminaries}
\setcounter{section}{2}\setcounter{equation}{0}

\begin{lemma}\label{lem2.3} $($Hardy--Littlewood--Sobolev theorem$)$ \\
Let  $s$  be a real number, with $0 < s < n$. The Riesz potential operator of order $s$ is $	I_s=(\Delta)^{-\frac{s}{2}}$, and let $1 < p < q < \infty$  satisfy $	\frac{1}{p}=\frac{1}{q}+\frac{s}{n}$, then there exist constants $C(n,s, p)$ such that for all $f$ in $\mathcal{F}(R^n)$  we
have
\begin{align*}
	\|I_s(f)\|_{L^q}\le C(n,s, p)	\|	f\|_{L^P}.
\end{align*}
\end{lemma}
\begin{lemma}\label{lem2.4} {\bf\cite{GHZ}}\\
Assume $f$ is a scalar function, $V$ is a vector-valued function (or scalar function)
and $M(\xi_1, \xi_2)$ is a matrix symbol (or scalar symbol). Define the bilinear operator
\begin{align*}
	O[f,M] g=\int_{R^{3}} e^{ix(\xi_1+\xi_2)}m(\xi_1, \xi_2)\widehat f(\xi_1)\widehat g(\xi_2)d\xi_1d\xi_2,
\end{align*}
then we have
\begin{align*}
\|O[f,M] g\|_{L^{2}}\lesssim \|m(\xi, \eta-\xi)\|_{L_{\eta}^{\infty}{H_{\xi}^{1}}}\|f \|_{L^{\infty}}\| g\|_{L^{2}}.
\end{align*}
\end{lemma}
\begin{lemma}\label{lem2.5}{\bf\cite{G}}\\
Coifman--Meyer operators are defined via a Fourier multiplier $m(\xi, \eta)$
\begin{align*}
	T_{m(\xi,\eta)}(f,g)=\int m(\xi, \eta)\widehat f(\xi-\eta)\widehat g(\eta)d\eta .
\end{align*}
If $p$, $q$, $r$ satisfy $\frac{1}{r}=\frac{1}{p}+\frac{1}{q}$, and if $\epsilon > 0 $ then
\begin{align*}
	\|T_{m(\xi,\eta)}(f,g)\|_{L^{r}}\lesssim \|m(\xi, \eta-\xi)\|_{{H_{\frac{3}{2}+\epsilon}}}\|f \|_{L^{p}}\| g\|_{L^{q}}.
\end{align*}
\end{lemma}
\begin{lemma}\label{lem2.6}$($Linear dispersive estimates$)$ {\bf\cite{ZFJ}}\\
A class of   the Schr$\ddot{o}$dinger semi--group  is $e^{it\Delta}$. Note that from the linear estimates for the Schr$\ddot{o}$dinger group,
\begin{align}
	&\|e^{it\Delta}f\|_{L^{6}}\lesssim \frac{1}{t}	\|xf\|_{L^{2}},\quad \quad\|e^{it\Delta}f\|_{L^{\infty}}\lesssim \frac{1}{t^{3/2}}	\|x^{2}f\|_{L^{2}}^{1/2}\|xf\|_{L^{2}}^{1/2}\lesssim\frac{1}{t^{1+\alpha}},\label{2.1}\\
&\|e^{it\Delta}f\|_{L^{\infty}}\lesssim \frac{1}{t^{3/2}}	\|f\|_{L^{1}},\quad \|e^{it\Delta}f\|_{L^{p}}\lesssim \frac{1}{t^{3/2}}	\|f\|_{L^{p{'}}}.\label{2.2}
\end{align}
Moreover, a class of dispersive semi--group is $e^{it|\alpha\nabla|}$,  by the linear dispersive estimate for the wave equation,
\begin{align}
	\|e^{it|\alpha\nabla|}g\|_{B^{s}_{p,r}}\lesssim \frac{1}{t^{1-\frac{2}{p}}}\|g\|_{B^{s}_{p{'},r}}^{2(1-2/p)}.\label{2.3}	
\end{align}
We note that by \eqref{2.3} with $r = 2$, and embeddings between Besov and Sobolev
spaces, we have
\begin{align}
	\|e^{it|\alpha\nabla|}g\|_{L^{p}}\lesssim \frac{1}{t^{1-\frac{2}{p}}}\|\wedge^{2(1-2/p)}g\|_{L^{p{'}}}.\label{2.4}
\end{align}
\end{lemma}

\begin{lemma}\label{lem7.1}
	For $\|(\mathcal{E},N,M)\|_{X}\lesssim \mathcal{E}_{0}$, where $n$ is the dimension, $p(|\nabla|)=(\Delta^{2}-\Delta)^{1/2}$, we have	
	\begin{align*}
	 \big\|e^{itp(|\nabla|)}H\big\|_{L^{\infty}}\lesssim t^{-\frac{3}{2}}\|\wedge^{\frac{1}{2}}H\|_{L^{1}}.
	\end{align*}
\end{lemma}
\begin{proof}
According to the define of $H$ in \eqref{4.10}
\begin{align}
	H(t,\xi)&=\mathcal{F}^{-1}\int_{0}^{t}\int_{R^{3}}\beta|\xi|(1+|\xi|^{2})^{-1/2}|\xi|^{2} e^{is\gamma_{\pm}(\xi,\eta)}\widehat f(\xi-\eta,s)\overline{\widehat  f}(\eta,s) d\eta ds,\nonumber\\
	&=\mathcal{F}^{-1}\int_{0}^{t}\int_{R^{3}}\beta|\xi|(1+|\xi|^{2})^{-1/2}|\xi|^{2} e^{ \mp is\beta|\eta|\sqrt{{\eta}^{2}+1}}e^{-is |\xi-\eta|^2}\widehat f(\xi-\eta,s)e^{is |\eta|^2}\overline{\widehat  f}(\eta,s) d\eta ds.\label{2.6}
\end{align}
The linear dispersion relation for \eqref{2.6} behaves like
\begin{align}
	p(\xi)=|\xi|\sqrt{|\xi|^{2}+1}.\label{2.7}
\end{align}
These solutions can be expressed in terms of the initial datas and of one half-wave operator $	T_{t}=e^{itp(|\nabla|)}$, for $p$ defined in \eqref{2.7} that we now study. Define $	p(r)=r\sqrt{r^{2}+1}$, therefore
\begin{align*}
	p{'}(r)=\frac{2r^{2}+1}{\sqrt{r^{2}+1}},\quad\quad\quad
	p{''}(r)=\frac{r(2r^{2}+3)}{({r^{2}+1})^{3/2}},\quad\quad\quad
	p{'''}(r)=\frac{3}{(1+r^{2})^{5/2}}.
\end{align*}
We note $p(r), p{'}(r), p{''}(r), p{'''}(r)$ that have no unique positive root, and when $r\ge1$,
\begin{align*}
	p{'}(r)=\frac{2r^{2}+1}{\sqrt{r^{2}+1}}\sim r=r^{m_{1}-1},\quad
	p{''}(r)=\frac{r(2r^{2}+3)}{({r^{2}+1})^{3/2}}\sim 1=r^{\alpha_{1}-2},
\end{align*}
we take  $m_{1}=2$, $m_{2}=1$. And when $r < 1$,
\begin{align*}
	p{'}(r)=\frac{2r^{2}+1}{\sqrt{r^{2}+1}}\sim 1=r^{m_{2}-1},\quad
	p{''}(r)=\frac{r(2r^{2}+3)}{({r^{2}+1})^{3/2}}\sim r=r^{\alpha_{2}-2},
\end{align*}
we take  $\alpha_{1}=2$, $\alpha_{2}=3$. Define: $\psi_{0}$ and $\psi_{\infty}$ are  smooth functions such that $0\le \psi_{0}+\psi_{\infty}=1$, further more
\begin{equation*}
	\left\{\!\!
	\begin{array}{lc}
		\psi_{0}(r)=1,|r|\le \epsilon,
		&\\ \psi_{0}(r)=0,|r|\ge \epsilon. &
	\end{array}
	\right.
\end{equation*}
\begin{equation*}
	\left\{\!\!
	\begin{array}{lc}
		\psi_{\infty}(r)=1,|r|\ge \epsilon,
		&\\ \psi_{\infty}(r)=0,|r|\le \epsilon. &
	\end{array}
	\right.
\end{equation*}
Case 1: according Theorem 1  in \cite{ZLB}, satisfies $H_{4}$, when $\theta=1$, we have
\begin{align*}
	\|e^{itp(r)}\psi_{0}(r)H\|_{L^{\infty}}&\lesssim t^{-\frac{n}{2}}2^{k(n-\frac{m_{1}(n-1+\theta)}{2}-\frac{\theta(\alpha_{1}-m_{1})}{2})}\|\psi_{0}(r)H\|_{L^{1}}
	\lesssim t^{-\frac{n}{2}}\|H\|_{L^{1}}.
\end{align*}
Case 2: according Theorem 1  in \cite{ZLB}, satisfies $H_{3}$, when $\theta=1$, we have
\begin{align*}
	\|e^{itp(r)}\psi_{0}(r)H\|_{L^{\infty}}&\lesssim t^{-\frac{n}{2}}2^{k(n-\frac{m_{2}(n-1+\theta)}{2}-\frac{\theta(\alpha_{2}-m_{2})}{2})}\|\psi_{0}(r)H\|_{L^{1}}
	\lesssim t^{-\frac{n}{2}}\|\wedge^{\frac{1}{2}}H\|_{L^{1}}.
\end{align*}
In this artical $n=3$, as a consequence, $\|e^{i\beta(\Delta^{2}-\Delta)^{1/2}t}H\|_{L^{\infty}}\lesssim t^{-\frac{3}{2}}\|\wedge^{\frac{1}{2}}H\|_{L^{1}}$.
\end{proof}
\begin{lemma}\label{lem2.7}
A class of dispersive semi--group is $e^{it\beta(\Delta^{2}-\Delta)^{1/2}}$. In view of Riesz--Thorin{'}s interpolation theorem, we have
\begin{align}
	\big\|	e^{it\beta(\Delta^{2}-\Delta)^{1/2}} h\big\|_{L^{p}}\lesssim \frac{1}{t^{1-\frac{2}{p}}}\big\|\wedge^{(1-2/p)}h\big\|_{L^{p{'}}}.\label{2.5}
\end{align}
\end{lemma}
\begin{proof}
According Theorem 1 in \cite{ZLB}, we have the following inequalities:
   \begin{align*}
 	\big\|	e^{it\beta(\Delta^{2}-\Delta)^{1/2}} f\big\|_{L^{\infty}}&\le \frac{1}{t}2^{k}\|f\|_{L^{1}},\\
 	\big\|	e^{it\beta(\Delta^{2}-\Delta)^{1/2}} f\big\|_{L^{2}}&\le \|f\|_{L^{2}}.
 \end{align*}
Then we use the Riesz--Thorin theorem, we  directly get
    \begin{align*}
 	\big\|	e^{it\beta(\Delta^{2}-\Delta)^{1/2}} f\big\|_{L^{p}}&\le (\frac{1}{t}2^{k})^{1-\frac{2}{p}}\|f\|_{L^{q}}\le \frac{1}{t^{1-\frac{2}{p}}}2^{k(1-\frac{2}{p})}\|f\|_{L^{q}}\le \frac{1}{t^{1-\frac{2}{p}}}\|\wedge^{1-\frac{2}{p}}f\|_{L^{q}}.
 \end{align*}
\end{proof}

%
%
%
%
%

%
%
%
%
%

\section{Resonance analysis and resolution space}

\setcounter{section}{3}\setcounter{equation}{0}

We write system \eqref{3.1} into integral forms. Define
\begin{equation}
	f:= e^{-it \Delta}\mathcal{E}, \quad\quad  g_{\pm}:=e^{{\pm}it|\alpha\nabla|}N_{\pm},\quad\quad h_{\pm}:=e^{{\pm}it\beta(\Delta^{2}-\Delta)^{1/2}}M_{\pm},\\
\end{equation}
denote the profiles. Then we have the following forms:
\begin{align}\label{4.2}
	&\widehat {\mathcal{E}}(\xi,t)=e^{-it |\xi|^{2}}\widehat f(\xi,t),\quad
	\widehat N_{\pm}(\xi,t)=e^{\pm it |\alpha\xi|}\widehat g(\xi,t),\quad
    \widehat M_{\pm}(\xi,t)=e^{\pm it|\xi|^{2}\sqrt{|\xi|^{2}+1}}\widehat h(\xi,t).
\end{align}
We use  Duhamel{'}s  formula as follows:
\begin{align*}
	\begin{split}	
		f_{t}=-i\Delta \mathcal{E}e^{-it \Delta}+e^{-it \Delta}\mathcal{E}_{t}=\frac{1}{2i}e^{-it \Delta}(N_{+}-N_{-})\mathcal{E}-\frac{1}{2}e^{-it \Delta}(M_{+}-M_{-})\mathcal{E},
	\end{split}
\end{align*}
thus
\begin{align}
	\begin{split}			
\widehat{f}(t,\xi)=\widehat{f(0,\xi)}+\int_{0}^{t}\frac{1}{2i}e^{is|\xi|^{2}}(\widehat{N_{+}\mathcal{E}}
-\widehat{N_{-}\mathcal{E}})-\frac{1}{2}e^{is|\xi|^{2}}(\widehat{M_{+}\mathcal{E}}-\widehat{M_{-}\mathcal{E}})ds.
	\end{split}
\end{align}
Finally, we  get
\begin{align}\label{4.4}
	\begin{split}
			&f(t,x)=f(0,x)+\sum_{\pm}\mathcal{F}^{-1}\frac{1}{2i}\int_{0}^{t}\int_{R^{3}}e^{is(|\xi|^{2}-|\xi-\eta|^{2}\pm|\alpha\eta|)}\widehat f(\xi-\eta,s)\widehat g_{\pm}(\eta,s)d\eta ds\\
			&\quad\quad\quad\quad\quad\quad\quad-\sum_{\pm}\mathcal{F}^{-1}\frac{1}{2}\int_{0}^{t}\int_{R^{3}}e^{is(|\xi|^{2}-|\xi-\eta|^{2}\pm \beta |\eta|\sqrt{{\eta}^{2}+1})}\widehat f(\xi-\eta,s)\widehat h_{\pm}(\eta,s)d\eta ds.\\
	\end{split}
\end{align}
Similarly, for the prolile $g_{\pm}=e^{{\pm}it|\alpha\nabla|}N_{\pm}$,  we  get
\begin{align}
		g_{\pm}(t,x)&=g(0,x)+\mathcal{F}^{-1}\frac{1}{i}\int_{0}^{t}\int_{R^{3}}e^{is(\mp |\alpha\xi|-|\xi-\eta|^{2}+|\eta|^{2})}|\alpha\xi|\widehat f(\xi-\eta,s)\overline{\widehat f}(\eta,s)d\eta ds.\label{4.5}
\end{align}
Finally,  for the profile  $h_{\pm}=e^{{\pm}i\beta(\Delta^{2}-\Delta)^{1/2}t}M_{\pm}$, we  get
\begin{align}	
	h_{\pm}(t,x)=h(0,x)+\mathcal{F}^{-1}\beta\int_{0}^{t}\int_{R^{3}}\frac{|\xi|}{(1+|\xi|^{2})^{1/2}}|\xi|^{2} e^{is(\mp \beta|\xi|^{2}\sqrt{\xi^{2}+1}-|\xi-\eta|^{2}+|\eta|^{2})}\widehat f(\xi-\eta,s)\overline{\widehat f}(\eta,s)d\eta ds.	\label{4.6}				
\end{align}
Define
\begin{align}
	&F_{1{\pm}}(t,x)=\mathcal{F}^{-1}\frac{1}{2i}\int_{0}^{t}\int_{R^{3}}e^{is\phi_{\pm}(\xi,\eta)}\widehat f(\xi-\eta,s)\widehat g_{\pm}(\eta,s)d\eta ds,\label{4.7}\\
	&F_{2{\pm}}(t,x)=\mathcal{F}^{-1}\frac{1}{2}\int_{0}^{t}\int_{R^{3}}e^{is\rho_{\pm}(\xi,\eta)}\widehat f(\xi-\eta,s)\widehat h_{\pm}(\eta,s)d\eta ds,\label{4.8}\\
	&G_{\pm}(t,x)=\mathcal{F}^{-1}\frac{1}{i}\int_{0}^{t}\int_{R^{3}}|\alpha\xi|e^{is\psi_{\pm}(\xi,\eta)}\widehat f(\xi-\eta,s)\overline{\widehat f}(\eta,s)d\eta ds,\label{4.9}\\
	&H_{\pm}(t,x)=\mathcal{F}^{-1}\beta\int_{0}^{t}\int_{R^{3}}\frac{|\xi|}{(1+|\xi|^{2})^{1/2}}|\xi|^{2} e^{is\gamma_{\pm}(\xi,\eta)}\widehat f(\xi-\eta,s)\overline{\widehat f}(\eta,s) d\eta ds,	\label{4.10}		
\end{align}
where
\begin{align}
		&\phi_{\pm}(\xi,\eta)=|\xi|^{2}-|\xi-\eta|^{2}\pm |\alpha\eta|=2\xi\eta -|\eta|^{2}\pm |\alpha\eta|,\label{4.11}\\
		&\rho_{\pm}(\xi,\eta)=|\xi|^{2}-|\xi-\eta|^{2}\pm \beta|\eta|\sqrt{{\eta}^{2}+1}=2\xi\eta -|\eta|^{2}\pm \beta|\eta|\sqrt{{\eta}^{2}+1},\label{4.12} \\
		&\psi_{\pm}(\xi,\eta)=\mp |\alpha\xi|-|\xi-\eta|^{2}+|\eta|^{2}=\mp |\alpha\xi|- |\xi|^{2}+2\xi\eta, \label{4.13} \\
		&\gamma_{\pm}(\xi,\eta)=\mp \beta |\xi|\sqrt{{\xi}^{2}+1}-|\xi-\eta|^{2}+|\eta|^{2}=\mp \beta |\xi|\sqrt{{\xi}^{2}+1}- |\xi|^{2}+2\xi\eta.\label{4.14}
\end{align}
Then \eqref{4.4}--\eqref{4.6}  could be shortly writen as
\begin{align*}
	&f(t,x)=f(0,x)+F_{1{\pm}}+F_{2{\pm}},\\
	&g_{\pm}(t,x)=g(0,x)+G_{\pm},\\
	&h_{\pm}(t,x)=h(0,x)+H_{\pm}.		
\end{align*}
with the initial datas as
\begin{align*}
	&f(0,x)=\mathcal{E}(0,x)=\mathcal{E}_{0}(x),\\
	&g(0,x)=N(0,x)=n(0,x)+i|\alpha\nabla|^{-1}n_{t}(0,x)=n_{0}(x)+i|\alpha\nabla|^{-1}n_{1}(x),	\\
	&h(0,x)=M(0,x)=\mathcal{B}(0,x)+i\beta^{-1}	(\Delta^{2}-\Delta)^{-1/2}\mathcal{B}_{t}(0,x)=\mathcal{B}_{0}(x)+i\beta^{-1}	(\Delta^{2}-\Delta)^{-1/2}\mathcal{B}_{1}(x).	
\end{align*}
In the following, we mainly consider \eqref{4.7}--\eqref{4.10}. The cases $\pm$ will be treated identically in our analysis. Therefore for
ease of exposition we will drop the apex $\pm$.

 \subsection{Space-time resonance set}
It becomes clear that the  properties of $\phi(\xi,\eta),\rho(\xi,\eta),\psi(\xi,\eta),\gamma(\xi,\eta)$ will provide a key to understand the scattering of the system \eqref{3.1}: this is the idea of space--time resonance. According to the concept of space--time resonance set in \cite{PN}, the space--time resonance set about $\phi_{\pm}$ is
\begin{align*}
	&\mathcal{T}_{\phi_{\pm}}=\left\{(\xi,\eta):\phi_{\pm}(\xi,\eta)=0\right\},\\
	&\mathcal{S}_{\phi_{\pm}}=\left\{(\xi,\eta):\nabla_{\eta} \phi_{\pm}(\xi,\eta)=0\right\}=\left\{2\xi-2|\eta|\pm |\alpha|=0\right\},\\
	&\mathcal{R}_{\phi_{\pm}}=\left\{(\xi,\eta):\mathcal{T}_{\phi_{\pm}}\cap \mathcal{S}_{\phi_{\pm}} \right\}=\left\{(\xi,\eta):\eta=0, |\xi|=\frac{1}{2}\alpha\right\}.
\end{align*}
The space--time resonance set about $\rho_{\pm}$ is
\begin{align*}
		&\mathcal{T}_{\rho_{\pm}}=\left\{(\xi,\eta):\rho_{\pm}(\xi,\eta)=0\right\},\\
		&\mathcal{S}_{\rho_{\pm}}=\left\{(\xi,\eta):\nabla_{\eta} \rho_{\pm}(\xi,\eta)=0\right\}=\left\{2\xi-2|\eta|\mp \beta \sqrt{{\eta}^{2}+1}\mp \beta |\eta|\frac{2\eta}{\sqrt{{\eta}^{2}+1}}=0\right\},\\
		&\mathcal{R}_{\rho_{\pm}}=\left\{(\xi,\eta):\mathcal{T}_{\rho_{\pm}}\cap \mathcal{S}_{\rho_{\pm}} \right\}=\left\{(\xi,\eta):\eta=0, |\xi|=\frac{1}{2}\beta\right\}.
\end{align*}
The space--time resonance set  about $\psi_{\pm}$ is
\begin{align*}
	&\mathcal{T}_{\psi_{\pm}}=\left\{(\xi,\eta):\psi_{\pm}(\xi,\eta)=0\right\},\\
	&\mathcal{S}_{\psi_{\pm}}=\left\{(\xi,\eta):\nabla_{\eta} \psi_{\pm}(\xi,\eta)=0\right\}=\left\{2\xi=0\right\},\\
	&\mathcal{R}_{\psi_{\pm}}=\left\{(\xi,\eta):\mathcal{T}_{\psi_{\pm}}\cap \mathcal{S}_{\psi_{\pm}} \right\}=\left\{(\xi,\eta):\eta=0, \xi=0\right\}.
\end{align*}
The space--time resonance set  about $\gamma_{\pm}$ is
\begin{align*}
		&\mathcal{T}_{\gamma_{\pm}}=\left\{(\xi,\eta):\gamma_{\pm}(\xi,\eta)=0\right\},\\
		&\mathcal{S}_{\gamma_{\pm}}=\left\{(\xi,\eta):\nabla_{\eta} \gamma_{\pm}(\xi,\eta)=0\right\}=\left\{2\xi=0\right\},\\
		&\mathcal{R}_{\gamma_{\pm}}=\left\{(\xi,\eta):\mathcal{T}_{\gamma_{\pm}}\cap \mathcal{S}_{\gamma_{\pm}} \right\}=\left\{(\xi,\eta):\eta=0, \xi=0\right\}.
\end{align*}
\subsection{Resolution space}
We define the resolution space $X$ by
\begin{align}\label{4.21}
\begin{split}
  &\big\|(\mathcal{E}=e^{-it \Delta}f,N_{\pm}=e^{{\mp}it|\alpha\nabla|}g_{\pm},M_{\pm}=e^{{\mp}it\beta(\Delta^{2}-\Delta)^{1/2}}h_{\pm})\big\|_{X}=\\& \underset{t}{sup} \Big(t^{-\delta}\|f(t)\|_{H^{N+1}}
	+t^{-\delta}\|xf(t)\|_{L^{2}}	+t^{-1+2\alpha+\delta}\|x^{2}f(t)\|_{L^{2}}+t^{1+\alpha}\big\|e^{it\Delta}f(t)\big\|_{L^{\infty}}\\
	&\quad+\|g_{\pm}(t)\|_{H^{N}}+\|xg_{\pm}(t)\|_{H^{1}}	+t^{-1+3\alpha}\|\wedge x^{2}g_{\pm}(t)\|_{L^{2}}+t\big\|e^{{\mp}it|\alpha\nabla|}g_{\pm}(t)\big\|_{B_{\infty,1}^{0}}\\
    &\quad+\|h_{\pm}(t)\|_{H^{N-1}}+\|xh_{\pm}(t)\|_{H^{1}}	+t^{-1+3\alpha}\|\wedge x^{2}h_{\pm}(t)\|_{L^{2}}+t\big\|e^{{\mp}it\beta(\Delta^{2}-\Delta)^{1/2}}h_{\pm}(t)\big\|_{L^{\infty}}\Big).
\end{split}
\end{align}
In the following, our goal is to  demonstrate that
\begin{align*}
\big\|e^{it \Delta}F_{\pm},e^{{\mp}it|\alpha\nabla|}G_{\pm},e^{{\mp}it\beta(\Delta^{2}-\Delta)^{1/2}}H_{\pm}\big\|_{X}\lesssim \|(\mathcal{E},N_{\pm},M_{\pm})\|_{X}^2,
\end{align*}
under  the hypothesis for initial datas in Theorem \ref{1.1}, that is
 \begin{align*}
 		\big\|e^{it \Delta}f(0),e^{{\mp}it|\alpha\nabla|}g(0),e^{{\mp}it\beta(\Delta^{2}-\Delta)^{1/2}}h(0)\big\|_{X}\lesssim \epsilon_{0}.
\end{align*}
 As a consequence, it implies the inequality
 \begin{align*}
		\big\|(\mathcal{E}=e^{it \Delta}f,N_{\pm}=e^{{\mp}it|\alpha\nabla|}g_{\pm},M_{\pm}=e^{{\mp}it\beta(\Delta^{2}-\Delta)^{1/2}}h_{\pm})\big\|_{X}\lesssim \epsilon_{0} +\|(\mathcal{E},N_{\pm},M_{\pm})\|_{X}^2.
\end{align*}
It is extended to a global result by the  priori estimates. Meanwhile, scattering result follows.

%
%
%
%
%

\section{Energy Estimates} \setcounter{section}{4}\setcounter{equation}{0} \label{sec1.1}

In this section we are going to prove the following:
\begin{proposition}\label{prop5.1}
	 $\|G\|_{H^{N}}+t^{-\delta}\|F\|_{H^{N+1}}+\|H\|_{H^{N-1}}\lesssim \|(\mathcal{E},N,M)\|_{X}^2 $.
\end{proposition}
\noindent {\bf Proof of Proposition \ref{prop5.1}}.\\
We then write $F=F_{1}+F_{2}$, $F_{1}$ and $F_{2}$ are defined in \eqref{4.7} and \eqref{4.8}. To prove this Proposition \ref{prop5.1}, we first prove the following lemma.
\begin{lemma}\label{lem5.1}
Let space X be defined in \eqref{4.21},  we have
\begin{align*}
\|G\|_{H^{N}}+t^{-\delta}\|F_{1}\|_{H^{N+1}}\lesssim  \|(\mathcal{E},N,M)\|_{X}^{2}.
\end{align*}
\end{lemma}
The proof of Lemma \ref{lem5.1} can be found in \cite{ZFJ} in Section 4.

{\bf Estimate of $\|F_2\|_{H^{N+1}}\lesssim t^{\delta}\|(\mathcal{E},N,M)\|_{X}^{2} $:} \\
Let us define smooth positive radial cutoff functions $X_{1}$ and $X_{2}$, with
$X_{1} +X_{2}=1$,
\begin{equation*}
		\left\{\!\!
		\begin{array}{lc}
		X_{2} =1,\quad if\quad  100|\xi-\eta|\le |\eta|,
		&\\ X_{2} =0,\quad if\quad  |\eta|\le 50|\xi-\eta|. &
		\end{array}
		\right.
\end{equation*}
We then write $F_{2}=F_{21}+F_{22}$. Define
\begin{align}
&F_{21}(t,\xi)=\mathcal{F}^{-1}\frac{1}{2}\int_{0}^{t}\int_{R^{3}}X_{1}e^{is\rho_{\pm}(\xi,\eta)}\widehat f(\xi-\eta,s)\widehat h(\eta,s)d\eta ds, \label{5.1}\\
&F_{22}(t,\xi)=\mathcal{F}^{-1}\frac{1}{2}\int_{0}^{t}\int_{R^{3}}X_{2}e^{is\rho_{\pm}(\xi,\eta)}\widehat f(\xi-\eta,s)\widehat h(\eta,s)d\eta ds.\label{5.2}
\end{align}
For the estimate of $F_{21}$, on the support of $X_{1}$, we have $|\xi-\eta|\gtrsim|\eta|$. It's High $\times$ Low $\rightarrow$ High case which hence derivatives applied to $F_{21}$  fall only on $f= e^{-it \Delta}\mathcal{E}$. By lemma \ref{lem2.4}, we have
\begin{align*}
	\begin{split}
		\|F_{21}\|_{H^{N+1}}		
		&\lesssim \int_{0}^{t}\Big\|X_{1}\frac{|\xi|^{N+1}}{|\xi-\eta|^{N+1}}\Big\|_{L_{\eta}^{\infty}H_{\xi}^{1}}\|\mathcal{E}\|_{H^{N+1}}\|M\|_{L^{\infty}}ds\\
		&\lesssim \int_{0}^{t} s^{\delta}\frac{1}{{\left\langle s \right\rangle}}ds\|(\mathcal{E},N,M)\|_{X}^2\\
&\lesssim t^{\delta}\|(\mathcal{E},N,M)\|_{X}^2,
	\end{split}
\end{align*}
here we use the fact that  $\partial_{\xi}(X_{1}\frac{|\xi|^{N+1}}{|\xi-\eta|^{N+1}})\lesssim C$. For the estimate of $F_{22}$, we can reduce ourselves to the case $|\eta|\ge 100$. Then we observe that on the suppport of $X_{2}$, when $\beta \neq 1$, the phase $\phi_{\pm}(\xi,\eta)$  satisfies the following:
\begin{align}
	|\rho_{\pm}(\xi,\eta)|\ge|\xi|^2 -|\xi-\eta|^{2}\pm \beta|\eta|\sqrt{{\eta}^{2}+1}\gtrsim |\eta|^{2}.
\end{align}
 It's High $\times$ High $\rightarrow$ Low case which hence derivatives applied to $F_{22}$  fall only on $h=e^{it\beta(\Delta^{2}-\Delta)^{1/2}}M_{\pm}$. Because of the mass conservation  of the magetic Zakharov system,  so $\|F\|_{L^{2}}\lesssim 1$. As a consequence,   we just need to consider $\|F_{22}\|_{H^{N+1}}$. Since $\|F_{22}\|_{H^{N+1}}=\|\Lambda^{N+1}F_{22}\|_{L^{2}}$, we have
\begin{align}
	\Lambda^{N+1} F_{22}(t,x)=\mathcal{F}^{-1}\int_{0}^{t}\int_{R^{3}}|\xi|^{N+1}X_{2}e^{is\rho(\xi,\eta)}\widehat f(\xi-\eta,s)\widehat h(\eta,s)d\eta ds.
\end{align}
Integrate by parts for s, we have
\begin{align}
	\Lambda^{N+1} F_{22}(t,x)
	&=\mathcal{F}^{-1}\int_{R^{3}}\frac {X_{2}|\xi|^{N+1}}{i|\eta|^{N-1}\rho(\xi,\eta)}e^{is\rho(\xi,\eta)}\widehat f(\xi-\eta,s)|\eta|^{N-1}\widehat h(\eta,s)d\eta \Big|_{0}^{t},\label{5.5} \\
	&-\mathcal{F}^{-1}\int_{R^{3}}\int_{0}^{t}\frac {X_{2}|\xi|^{N+1}}{i|\eta|^{N-1}\rho(\xi,\eta)}e^{is\rho(\xi,\eta)}\partial_{s}\widehat f(\xi-\eta,s)|\eta|^{N-1}\widehat h(\eta,s)dsd\eta , \label{5.6}\\
	&-\mathcal{F}^{-1}\int_{R^{3}}\int_{0}^{t}\frac {X_{2}|\xi|^{N+1}}{i|\eta|^{N-1}\rho(\xi,\eta)}e^{is\rho(\xi,\eta)}\widehat f(\xi-\eta,s)|\eta|^{N-1}\partial_{s}\widehat h(\eta,s)dsd\eta ,\label{5.7}
\end{align}
 on the support of $X_{2}$, we have
\begin{align}
	\bigg|\frac {X_{2}|\xi|^{N+1}}{i|\eta|^{N-1}\rho(\xi,\eta)}\bigg|\lesssim \bigg|\frac{|\xi|^{N+1}}{|\eta|^{N-1}|\eta|^{2}}\bigg|\lesssim 1,\quad\quad
	\partial_{\xi}\Big(\frac {X_{2}|\xi|^{N+1}}{i|\eta|^{N-1}\rho(\xi,\eta)}\Big)\lesssim C.\label{5.9}
\end{align}
For the estimate of  \eqref{5.5}, by lemma \ref{lem2.4} and  \eqref{5.9}, we have
\begin{align*}
		\|\eqref{5.5}\|_{L^{2}}
		&\lesssim \Big\|\frac {X_{2}|\xi|^{N+1}}{i|\eta|^{N-1}\rho(\xi,\eta)}\Big\|_{L_{\eta}^{\infty}H_{\xi}^{1}}\|\mathcal{E}\|_{L^{\infty}} \Big\|\Lambda^{N-1}e^{\mp it\beta(\Delta^{2}-\Delta)^{1/2}}h\Big\|_{L^{2}}\\
		&\lesssim \|\mathcal{E}\|_{L^{\infty}} \|M\|_{H^{N-1}}\\
&\lesssim t^{\delta}\|(\mathcal{E},N,M)\|_{X}^2  .
\end{align*}
For the estimate of \eqref{5.6}, from   $\mathcal{E}N+\mathcal{E}M= e^{it \Delta}\partial_{s}f$, Product laws, lemma \ref{lem2.4} and \eqref{5.9},  we obtain
 \begin{align*}
 		\|\eqref{5.6}\|_{L^{2}}					
 		&\lesssim \int_{0}^{t}\Big\|\frac {X_{2}|\xi|^{N+1}}{i|\eta|^{N-1}\rho(\xi,\eta)}\Big\|_{L_{\eta}^{\infty}H_{\xi}^{1}}\big\|e^{is \Delta}\partial_{s} f\big\|_{L^{\infty}}\Big\||\eta|^{N-1} e^{\mp it\beta(\Delta^{2}-\Delta)^{1/2}}h\Big\|_{L^{2}}ds\\
 		&\lesssim \int_{0}^{t}\|\mathcal{E}N+\mathcal{E}M\|_{L^{\infty}}\| M\|_{H^{N-1}}ds\\&\lesssim t^{\delta}\|(\mathcal{E},N,M)\|_{X}^{2}.	
 \end{align*}
Analogously, since $ \beta(-\Delta)^{1/2}(I-\Delta)^{-1/2}\nabla\times(\nabla\times(\mathcal{E}\times
\overline {\mathcal{E}}))=e^{i\beta(\Delta^{2}-\Delta)^{1/2}t}\partial_{s}h$, we have
\begin{align*}
		\|\eqref{5.7}\|_{L^{2}}
&\lesssim \int_{0}^{t}\Big\|\frac {X_{2}|\xi|^{N+1}}{i|\eta|^{N-1}\rho(\xi,\eta)}\Big\|_{L_{\eta}^{\infty}H_{\xi}^{1}}\big\|e^{is \Delta}\partial_{s} f\big\|_{L^{\infty}}\Big\||\eta|^{N-1} e^{\mp it\beta(\Delta^{2}-\Delta)^{1/2}}h\Big\|_{L^{2}}ds\\
		&\lesssim  \int_{0}^{t}\|\mathcal{E}\|_{L^{\infty}}\|\mathcal{E}\|_{H^{N+1}}\|\mathcal{E}\|_{L^{\infty}}ds\\&\lesssim t^{\delta}\|(\mathcal{E},N,M)\|_{X}^{2}.	
\end{align*}

{\bf Estimate of $\|H\|_{H^{N-1}}\lesssim \|(u,N,M)\|_{X}^{2}	 $:} \\
$H$ is defined in \eqref{4.10}. Just from H$\ddot{o}$lder{'}s inequality and lemma \ref{lem2.4},  we have
\begin{align*}
		\|H\|_{H^{N-1}}&\lesssim \int_{0}^{t}\Big\|\frac{|\xi|}{(1+|\xi|^{2})^{1/2}}\Big\|_{L_{\eta}^{\infty}H_{\xi}^{1}}\big\|e^{is \Delta} f\big\|_{L^{\infty}}\Big\||\eta|^{N+1} e^{is \Delta}f\Big\|_{L^{2}}ds\\		
		&\lesssim\int_{0}^{t} \|\mathcal{E}\|_{H^{N+1}}\|\mathcal{E}\|_{L^{\infty}}ds\\&\lesssim\|(\mathcal{E},N,M)\|_{X}^{2}.
\end{align*}
From lemma \ref{lem5.1} and the estimates of \eqref{5.1},\eqref{5.5}--\eqref{5.7},  the proof of Proposition \ref{prop5.1} is completed.

$\hfill\Box$\\
\noindent {\bf High frequency cutoff} \\
 We have established the priori
estimates $\|\mathcal{E}\|_{H^{N+1}}\lesssim t^{\delta}$, $\|N\|_{H^{N}}\lesssim 1$, $\|M\|_{H^{N-1}}\lesssim 1$. Let us denote by $P_{\ge k} $ the Littlewood--Paley projection on frequencies larger or equal to
$2^{k}$. Since for $k\ge 0$ one has then, for frequencies $\xi=2^{k}\gtrsim s^{p}$,  according to  Littlewood--Paley theory, we have
\begin{align}
	&\big\|P_{\ge k}\mathcal{E}(s)\big\|_{H_{3}}\lesssim \frac{1}{{\left\langle s \right\rangle}^2}s^{\delta},\\
	&\big\|P_{\ge k}N(s)\big\|_{H_{2}}\lesssim \frac{1}{{\left\langle s \right\rangle}^2},\\
	&\big\|P_{\ge k}M(s)\big\|_{H_{1}}\lesssim \frac{1}{{\left\langle s \right\rangle}^2}.
\end{align}
Thus, we choose $p(N-2)=2$, so $P=\frac{2}{N-2}$. We define: $\delta_{N}=\frac{2}{N-2}$. In the rest of the paper, we assume that all frequencies $|\xi-\eta|$ and $|\eta|$ appearing in the estimates of the bilinear terms, are bounded above by $s^{\delta_{N}} $, where $s^{\delta_{N}}:=\frac{2}{N-2}$ and the integer $N\ge 1$ is determined in the course of our proof
by several upperbounds on $s^{\delta_{N} } $. In particular, expressions such as $|\xi|$ and  $\nabla_{\xi}\psi (\xi,\eta)$, $\nabla_{\xi}\gamma (\xi,\eta)$ will
be often replaced by a factor of $s^{\delta_{N} } $.

%
%
%
%
%

\section{Weighted Estimates} \setcounter{section}{5}\setcounter{equation}{0}

In this section we are going to prove the following:
\begin{proposition}\label{prop6.1}
	$\|xG\|_{H^{1}}+t^{-1+3\alpha}\|\wedge x^{2}G\|_{L^{2}}\lesssim \|(\mathcal{E},N,M)\|_{X}^2 $.
\end{proposition}
For a proof of this proposition, we refer to \cite{ZFJ}.
\begin{proposition}\label{prop6.2}
	$\|xH\|_{H^{1}}+t^{-1+3\alpha}\|\wedge x^{2}H\|_{L^{2}}\lesssim \|(\mathcal{E},N,M)\|_{X}^2 $.
\end{proposition}
\noindent {\bf Proof of Proposition \ref{prop6.2}}.
More precisely we have
\begin{align}
	|\xi|=\frac{1}{2} \nabla_{\eta}\gamma (\xi,\eta),\label{6.1}
\end{align}
which allows us to integrate by parts  for $\eta$.\\
\noindent {\bf For the estimate of} $\|xH\|_{H^{1}}\lesssim \|(\mathcal{E},N,M)\|_{X}^{2} $.\\
 Applying $\nabla_{\xi}$ to $\widehat H$ gives the terms:
\begin{align}
	\nabla_{\xi}\widehat H&=\mathcal{F}^{-1}\int_{0}^{t} \int_{R^{3}}\beta e^{is\gamma(\xi,\eta)}\frac{|\xi|}{({1+|\xi|^{2}})^{1/2}}|\xi|^{2}\nabla_{\xi}\widehat f(\xi-\eta,s)\overline{\widehat  f}(\eta,s)d\eta ds\label{6.2},\\
	&+\mathcal{F}^{-1}\int_{0}^{t} \int_{R^{3}}\beta e^{is\gamma(\xi,\eta)}\frac{3|\xi|^2+2|\xi|^4}{({1+|\xi|^{2}})^{3/2}}\widehat f(\xi-\eta,s)\overline{\widehat  f}(\eta,s)d\eta ds\label{6.3},\\
	&+\mathcal{F}^{-1}\int_{0}^{t} \int_{R^{3}}\beta e^{is\gamma(\xi,\eta)}\frac{|\xi|}{({1+|\xi|^{2}})^{1/2}}|\xi|^{2}is\nabla_{\xi}\gamma(\xi,\eta) \widehat f(\xi-\eta,s)\overline{\widehat  f}(\eta,s)d\eta ds\label{6.4}.
\end{align}
For the estimate of \eqref{6.2}, from lemma \ref{lem2.4},  $\partial_{\xi}\Big(\frac{|\xi|}{({1+|\xi|^{2}})^{1/2}}|\xi|^{2}\Big)\lesssim s^{2\delta_{N}}$ and the priori estimate of $f$ \eqref{2.1}, we get
\begin{align*}
		\|\eqref{6.2}\|_{L^{2}}
		&\lesssim \beta\int_{0}^{t} \Big\| \frac{|\xi|}{({1+|\xi|^{2}})^{1/2}}|\xi|^{2}\Big\|_{L_{\eta}^{\infty}H_{\xi}^{1}}\big\| e^{is \Delta}xf\big\|_{L^{2}}\big\| e^{is \Delta}f\big\|_{L^{\infty}}ds\\
&\lesssim\int_{0}^{t}\beta s^{2\delta_{N}}s^{\delta}\frac{1}{{\left\langle s \right\rangle}^{1+\alpha}}ds\|(\mathcal{E},N,M)\|_{X}^{2}\\
		&\lesssim \beta|(\mathcal{E},N,M)\|_{X}^{2},
\end{align*}
 where we could choose $\alpha\ge \delta +2\delta_{N}$. Similarly,  thanks to
$\partial_{\xi}\big(\frac{3|\xi|^2+2|\xi|^4}{({1+|\xi|^{2}})^{3/2}}\big)\lesssim s^{\delta_{N}}$,
 we   get
\begin{align*}
		\|\eqref{6.3}\|_{L^{2}}	
		&\lesssim \beta \int_{0}^{t} \Big\| \frac{3|\xi|^2+2|\xi|^4}{({1+|\xi|^{2}})^{3/2}}|\xi|^{2}\Big\|_{L_{\eta}^{\infty}H_{\xi}^{1}}\big\| f\big\|_{L^{2}}\big\| e^{is \Delta}f\big\|_{L^{\infty}}ds
		\lesssim \|(\mathcal{E},N,M)\|_{X}^{2}.
\end{align*}
  Using the identity \eqref{6.1}  and integrating by parts for  $\eta$ in  \eqref{6.4},  we obtain the following terms:
 \begin{align}		
	&\int_{0}^{t}\frac{\beta}{2}\nabla_{\xi}\gamma(\xi,\eta) \frac{|\xi|^2}{({1+|\xi|^{2}})^{1/2}} e^{is\gamma(\xi,\eta)}
	\widehat f(\xi-\eta,s)\overline{\widehat  f}(\eta,s)  ds\Big|_{R^{3}},\label{6.5}\\
	&\int_{0}^{t}\int_{R^{3}}\frac{\beta}{2}\nabla_{\xi}\gamma(\xi,\eta) e^{is\gamma(\xi,\eta)} \frac{|\xi|^2}{({1+|\xi|^{2}})^{1/2}}\nabla_{\eta}\widehat f(\xi-\eta,s)\overline{\widehat  f}(\eta,s) d\eta ds,\label{6.6}\\
	&\int_{0}^{t}\int_{R^{3}}\frac{\beta}{2}\nabla_{\xi}\gamma(\xi,\eta) e^{is\gamma(\xi,\eta)} \frac{|\xi|^2}{({1+|\xi|^{2}})^{1/2}}\widehat f(\xi-\eta,s)\nabla_{\eta}\overline{\widehat  f}(\eta,s) d\eta ds.\label{6.7}
\end{align}
\eqref{6.5}=0 follows from Riemann--Lebesgue lemma. \eqref{6.6} and \eqref{6.7} are similar to \eqref{6.3}, we  have
\begin{align*}
\|\eqref{6.6}\|_{L^{2}}	+ \|\eqref{6.7}\|_{L^{2}}
		\lesssim \|(\mathcal{E},N,M)\|_{X}^{2}.
\end{align*}

\noindent {\bf For the estimate of}  $t^{-1+3\alpha}\|\wedge x^{2}H\|_{L^{2}}\lesssim \|(\mathcal{E},N,M)\|_{X}^{2} $.\\ Applying $|\xi|\nabla_{\xi}^{2}$ to $\widehat  H$ gives the following contributions:
\begin{align}
	|\xi|\nabla_{\xi}^{2}\widehat H&=\mathcal{F}^{-1}\int_{0}^{t}\int_{R^{3}}\beta\frac{|\xi|}{(1+|\xi|^{2})^{1/2}}|\xi|^{3} e^{is\gamma(\xi,\eta)}\nabla_{\xi}^{2}\widehat f(\xi-\eta,s)\overline{\widehat  f}(\eta,s) d\eta ds,\label{6.8}\\
	&+\mathcal{F}^{-1}\int_{0}^{t}\int_{R^{3}}\beta\frac{|\xi|}{(1+|\xi|^{2})^{1/2}}|\xi|^{3}is\nabla_{\xi}\gamma(\xi,\eta) e^{is\gamma(\xi,\eta)}\nabla_{\xi}\widehat f(\xi-\eta,s)\overline{\widehat  f}(\eta,s) d\eta ds,\label{6.9}\\
	&+\mathcal{F}^{-1}\int_{0}^{t}\int_{R^{3}}\beta\frac{|\xi|}{(1+|\xi|^{2})^{1/2}}|\xi|^{3}(is\nabla_{\xi}\gamma(\xi,\eta))^{2} e^{is\gamma(\xi,\eta)}\widehat f(\xi-\eta,s)\overline{\widehat  f}(\eta,s) d\eta ds,\label{6.10}\\
		&+\mathcal{F}^{-1}\int_{0}^{t}\int_{R^{3}}\beta\frac{6|\xi|^2+5|\xi|^4+2|\xi|^6}{(1+|\xi|^{2})^{5/2}} e^{is\gamma(\xi,\eta)}\widehat f(\xi-\eta,s)\overline{\widehat  f}(\eta,s) d\eta ds,\label{6.11}\\				
		&+\mathcal{F}^{-1}\int_{0}^{t}\int_{R^{3}}\beta\frac{3|\xi|^3+2|\xi|^5}{(1+|\xi|^{2})^{3/2}}e^{is\gamma(\xi,\eta)}\nabla_{\xi}\widehat f(\xi-\eta,s)\overline{\widehat  f}(\eta,s)  d\eta ds,\label{6.12}\\
		&+\mathcal{F}^{-1}\int_{0}^{t}\int_{R^{3}}\beta\frac{3|\xi|^3+2|\xi|^5}{(1+|\xi|^{2})^{3/2}}e^{is\gamma(\xi,\eta)}is\nabla_{\xi}\gamma(\xi,\eta) \widehat f(\xi-\eta,s)\overline{\widehat  f}(\eta,s) d\eta ds.\label{6.13}
\end{align}
From the fact that
$\partial_{\xi}\Big(\frac{|\xi|}{(1+|\xi|^{2})^{1/2}}|\xi|^{3}\Big)\lesssim s^{3\delta_{N}}$, \eqref{2.1} and lemma \ref{lem2.4} in \eqref{6.8}, we  get
\begin{align*}
	\|\eqref{6.8}\|_{L^{2}}
		&\lesssim \beta \int_{0}^{t} \Big\| \frac{|\xi|}{(1+|\xi|^{2})^{1/2}}|\xi|^{3}\Big\|_{L_{\eta}^{\infty}H_{\xi}^{1}}\big\| e^{is \Delta}x^2f\big\|_{L^{2}}\big\| e^{is \Delta}f\big\|_{L^{\infty}}ds\\
&\lesssim\int_{0}^{t}\beta s^{3\delta_{N}}s^{1-2\alpha-\delta}\frac{1}{{\left\langle s \right\rangle}^{1+\alpha}}ds\|(\mathcal{E},N,M)\|_{X}^{2}\\
	&\lesssim  t^{1-3\alpha}\|(\mathcal{E},N,M)\|_{X}^{2},
\end{align*}
where $3\delta_{N} \le\delta$.
Using the identity \eqref{6.1}, integrating by parts for $\eta$  in \eqref{6.9} and integrating by parts for $\eta$  in \eqref{6.10}, we have
\begin{align}		
	&\int_{0}^{t}\int_{R^{3}}\frac{\beta}{2}\frac{|\xi|^{3}}{(1+|\xi|^{2})^{1/2}}  \nabla_{\xi}\gamma(\xi,\eta) e^{is\gamma(\xi,\eta)}\nabla_{\eta}\nabla_{\xi}\widehat f(\xi-\eta,s)\overline{\widehat  f}(\eta,s)d\eta ds,\label{6.14}\\	
	&\int_{0}^{t}\int_{R^{3}}\frac{\beta}{2}\frac{|\xi|^{3}}{(1+|\xi|^{2})^{1/2}}  \nabla_{\xi}\gamma(\xi,\eta) e^{is\gamma(\xi,\eta)}\nabla_{\xi}\widehat f(\xi-\eta,s)\nabla_{\eta}\overline{\widehat  f}(\eta,s)d\eta ds,\label{6.15}\\
	&\int_{0}^{t}\int_{R^{3}}\frac{\beta}{2}\frac{|\xi|^{3}}{(1+|\xi|^{2})^{1/2}}is(\nabla_{\xi}\gamma(\xi,\eta))^{2} e^{is\gamma(\xi,\eta)} \nabla_{\eta}\widehat f(\xi-\eta,s)\overline{\widehat  f}(\eta,s)d\eta ds,\label{6.16}\\
	&\int_{0}^{t}\int_{R^{3}}\frac{\beta}{2}\frac{|\xi|^{3}}{(1+|\xi|^{2})^{1/2}}is(\nabla_{\xi}\gamma(\xi,\eta))^{2} e^{is\gamma(\xi,\eta)} \widehat f(\xi-\eta,s) \nabla_{\eta}\overline{\widehat  f}(\eta,s)d\eta ds.\label{6.17}
\end{align}
Due to the symmetry,   \eqref{6.17} is similar to \eqref{6.16}. We will only illustrate the estimate on \eqref{6.16}. Integrate by parts for $\eta$  in \eqref{6.16},
 \begin{align} &\int_{0}^{t}\int_{R^{3}}\frac{\beta}{4}\frac{|\xi|^{2}}{(1+|\xi|^{2})^{1/2}}(\nabla_{\xi}\gamma(\xi,\eta))^{2}e^{is\gamma(\xi,\eta)}\nabla_{\eta}^{2}\widehat f(\xi-\eta,s)\overline{\widehat  f}(\eta,s) d\eta ds,\label{6.18}\\	&\int_{0}^{t}\int_{R^{3}}\frac{\beta}{4}\frac{|\xi|^{2}}{(1+|\xi|^{2})^{1/2}}(\nabla_{\xi}\gamma(\xi,\eta))^{2}e^{is\gamma(\xi,\eta)}\nabla_{\eta}\widehat f(\xi-\eta,s)\nabla_{\eta}\overline{\widehat  f}(\eta,s)d\eta  ds.\label{6.19}
\end{align}
 \eqref{6.14} and  \eqref{6.18} give term like \eqref{6.8},  \eqref{6.19} gives term like \eqref{6.15}, which can be dealed with in a similar way, hence we have
 \begin{align*}
	\|\eqref{6.14}\|_{L^{2}}+\|\eqref{6.18}\|_{L^{2}}+\|\eqref{6.19}\|_{L^{2}}
	\lesssim  t^{1-3\alpha}\|(\mathcal{E},N,M)\|_{X}^{2}.
\end{align*}
 For the estimate of \eqref{6.15}, first, since
 \begin{align*}
 \partial_{\xi}^2\Big(\frac{|\xi|^{3}}{(1+|\xi|^{2})^{1/2}}(1+|\xi|^{2})^{-1/2}  \nabla_{\xi}\gamma(\xi,\eta)\Big)\lesssim s^{2\delta_{N}}.
 \end{align*}
Then, with the help of  the lemma \ref{lem2.5} and $\|f \|_{L^{\frac{4}{3}}}\lesssim \|f \|_{L^{2}}^{\frac{1}{4}}\|xf \|_{L^{2}}^{\frac{3}{4}}$, we  have
\begin{align*}
	\|\eqref{6.15}\|_{L^{2}}
	&\lesssim \int_{0}^{t}\frac{\beta}{2}\Big\| \frac{|\xi|^{3}}{(1+|\xi|^{2})^{1/2}} \nabla_{\xi}\gamma(\xi,\eta)\Big\|_{H_{\xi}^{2}}\big\| e^{is \Delta}xf\big\|_{L^{4}}\big\| e^{is \Delta}xf\big\|_{L^{4}}ds\\
	&\lesssim\int_{0}^{t}\frac{\beta}{2} s^{2\delta_{N}}\big\|  e^{is \Delta}xf\big\|_{L^{4}}^{2}ds\\&\lesssim\int_{0}^{t}\frac{\beta}{2} s^{2\delta_{N}}\frac{1}{{\left\langle s \right\rangle}^{3/2}}\| xf\|_{L^{4/3}}^{2}ds\\
	&\lesssim  t^{1-3\alpha}\|(\mathcal{E},N,M)\|_{X}^{2},
\end{align*}
where  $2\delta_{N}\le \delta$.  For the estimates of  \eqref{6.11} and \eqref{6.12}, they can be derived in a straightforward fashion which is identical to
\eqref{6.3}, so we directly get
\begin{align*}
	\|\eqref{6.11}\|_{L^{2}}+\|\eqref{6.12}\|_{L^{2}}
	\lesssim  t^{1-3\alpha}\|(\mathcal{E},N,M)\|_{X}^{2},
\end{align*}
where  we use the fact that \begin{align*} \partial_{\xi}\Big(\frac{6|\xi|^2+5|\xi|^4+2|\xi|^6}{(1+|\xi|^{2})^{5/2}}\nabla_{\xi}\gamma(\xi,\eta)\Big)\lesssim s^{\delta_{N}}, \quad \partial_{\xi}\Big(\frac{3|\xi|^3+2|\xi|^5}{(1+|\xi|^{2})^{3/2}}\nabla_{\xi}\gamma(\xi,\eta)\Big)\lesssim s^{2\delta_{N}}.
\end{align*} Finally,  using the identity \eqref{6.1}  and integrating by parts for $\eta$  in \eqref{6.13}, we have
\begin{align}
	&\int_{0}^{t}\int_{R^{3}}\frac{\beta}{2}\frac{3|\xi|^2+2|\xi|^4}{(1+|\xi|^{2})^{3/2}} \nabla_{\xi}\gamma(\xi,\eta) e^{is\gamma(\xi,\eta)}\nabla_{\eta}\widehat f(\xi-\eta,s)\overline{\widehat  f}(\eta,s)d\eta ds,\label{6.20}\\	
	&\int_{0}^{t}\int_{R^{3}}\frac{\beta}{2}\frac{3|\xi|^2+2|\xi|^4}{(1+|\xi|^{2})^{3/2}} \nabla_{\xi}\gamma(\xi,\eta) e^{is\gamma(\xi,\eta)}\widehat f(\xi-\eta,s)\nabla_{\eta}\overline{\widehat  f}(\eta,s)d\eta ds.\label{6.21}
\end{align}
\eqref{6.20}  and \eqref{6.21} are similar to \eqref{6.2}, so we  directly have
\begin{align*}
	\|\eqref{6.20}\|_{L^{2}}+\|\eqref{6.21}\|_{L^{2}}
	\lesssim  t^{1-3\alpha}\|(\mathcal{E},N,M)\|_{X}^{2},
\end{align*}
where $ \partial_{\xi}(\frac{4|\xi|^2+3|\xi|^4}{(1+|\xi|^{2})^{3/2}}\nabla_{\xi}\gamma(\xi,\eta))\lesssim s^{2\delta_{N}}$.
From  the estimates of \eqref{6.2}--\eqref{6.21},  the proof of Proposition \ref{prop6.2} is completed.

$\hfill\Box$

Here we show how to obtain some improved weighted $L^2$ estimates for $G$ and $H$. Since the existence of $\nabla_\xi\psi$ and $\nabla_\xi\gamma$, we can only obtain bad bounds on $x^2G_{\pm}$  and $x^2H_{\pm}$ in $L^2$. To deal with this difficulty, we  split $G_\pm$ and $H_\pm$ into two
parts based on physical space.

\begin{lemma}\label{lem6.1}
Let space X be defined in \eqref{4.21},	we have
\begin{align*}
t^{-3/4}\|x^{2}g_{1}\|_{L_{2}}+t^{(-7/4+3\delta)k} \|P_{\le k}g_{2}\|_{L_{2}}\lesssim \|(\mathcal{E},N,M)\|_{X}^2 .
\end{align*}
\end{lemma}
For a proof of this lemma, we refer to \cite{ZFJ}.

\begin{lemma}\label{lem6.2}
Let space X be defined in \eqref{4.21},	we have
\begin{align*}
t^{-3/4}\|x^{2}h_{1}\|_{L_{2}}+t^{(-11/4+3\delta)k} \|P_{\le k}h_{2}\|_{L_{2}}\lesssim \|(\mathcal{E},N,M)\|_{X}^2. \end{align*}
\end{lemma}
\begin{proof}
We split $H( f, \overline f ) $
\begin{align*}
H(f,\overline f):&=H(f_{\le s^{1/4}}+f_{\ge s^{1/4}},\overline f_{\le s^{1/4}}+f_{\ge s^{1/4}})\\
	&=H(f_{\le s^{1/4}},\overline f_{\le s^{1/4}})+H(f_{\le s^{1/4}},\overline f_{\ge s^{1/4}})
	+H(f_{\ge s^{1/4}},\overline f_{\ge s^{1/4}})+H(f_{\ge s^{1/4}},\overline f_{\le s^{1/4}}),
\end{align*}
then we have  two components
\begin{align*}
	h_{1}:=H(f_{\le s^{1/4}},\overline f_{\le s^{1/4}}),\quad\quad
	h_{2}:=H(f_{\le s^{1/4}},\overline f_{\ge s^{1/4}})+H(f_{\ge s^{1/4}},\overline f).
\end{align*}
The two terms in the definition of $h_{2}$ can be treated similarly, so we reduce to consider
$h_{1}$ and $h_{2}$ given by
\begin{align}
	&\widehat h_{1}:=\int_{0}^{t}\int_{R^{3}}\beta\frac{|\xi|}{(1+|\xi|^{2})^{1/2}}|\xi|^{2} e^{is\gamma_{\pm}(\xi,\eta)}\widehat{ f_{\le s^{1/4}}}(\xi-\eta,s)\overline{\widehat {f_{\le s^{1/4}}}}(\eta,s)d\eta ds, \label{6.22}\\
	&\widehat h_{2}:=\int_{0}^{t}\int_{R^{3}}\beta\frac{|\xi|}{(1+|\xi|^{2})^{1/2}}|\xi|^{2} e^{is\gamma_{\pm}(\xi,\eta)}\widehat {f_{\ge s^{1/4}}}(\xi-\eta,s)\overline{\widehat  f}(\eta,s)d\eta ds.\label{6.23}
\end{align}
\noindent {\bf For the estimate of} $\|x^{2}h_{1}\|_{L_{2}}\lesssim t^{3/4} $. \\Applying $\nabla_{\xi}^{2}$ to $\widehat h_{1}$ gives the terms:
\begin{align}
	\nabla_{\xi}^{2}\widehat h_{1}	=&\int_{0}^{t}\int_{R^{3}}\beta\frac{|\xi|}{(1+|\xi|^{2})^{1/2}}|\xi|^{2} e^{is\gamma_{\pm}(\xi,\eta)}\nabla_{\xi}^{2}\widehat {f_{\le s^{1/4}}}(\xi-\eta,s)\overline{\widehat {f_{\le s^{1/4}}}}(\eta,s)d\eta ds,\label{6.24}\\
	+&\int_{0}^{t} \int_{R^{3}}\beta\frac{|\xi|}{(1+|\xi|^{2})^{1/2}}|\xi|^{2} s\nabla_{\xi}\gamma(\xi,\eta) e^{is\gamma_{\pm}(\xi,\eta)}\nabla_{\xi}\widehat {f_{\le s^{1/4}}}(\xi-\eta,s)\overline{\widehat {f_{\le s^{1/4}}}}(\eta,s)d\eta ds,\label{6.25}\\	
	+&\int_{0}^{t} \int_{R^{3}}\beta\frac{|\xi|}{(1+|\xi|^{2})^{1/2}}|\xi|^{2} (s\nabla_{\xi}\gamma(\xi,\eta))^{2} e^{is\gamma_{\pm}(\xi,\eta)}\widehat {f_{\le s^{1/4}}}(\xi-\eta,s)\overline{\widehat {f_{\le s^{1/4}}}}(\eta,s)d\eta ds,\label{6.26}\\
	+&\int_{0}^{t} \int_{R^{3}}\beta\frac{6|\xi|+5|\xi|^3+2|\xi|^5}{(1+|\xi|^{2})^{5/2}} e^{is\gamma_{\pm}(\xi,\eta)}\widehat {f_{\le s^{1/4}}}(\xi-\eta,s)\overline{\widehat {f_{\le s^{1/4}}}}(\eta,s)d\eta ds,\label{6.27}\\
	+&\int_{0}^{t} \int_{R^{3}}\beta\frac{3|\xi|^2+2|\xi|^4}{(1+|\xi|^{2})^{3/2}} e^{is\gamma_{\pm}(\xi,\eta)}\nabla_{\xi}\widehat {f_{\le s^{1/4}}}(\xi-\eta,s)\overline{\widehat {f_{\le s^{1/4}}}}(\eta,s)d\eta ds\label{6.28},	\\
	+&\int_{0}^{t} \int_{R^{3}}\beta\frac{3|\xi|^2+2|\xi|^4}{(1+|\xi|^{2})^{3/2}}e^{is\gamma_{\pm}(\xi,\eta)}is\nabla_{\xi}\gamma(\xi,\eta)\widehat {f_{\le s^{1/4}}}(\xi-\eta,s)\overline{\widehat {f_{\le s^{1/4}}}}(\eta,s)d\eta ds\label{6.29}.	
\end{align}
For the estimates of \eqref{6.24} and \eqref{6.25}, first, by   the fact that
\begin{align}
	&\| x^{2}f_{\le s^{1/4}}\|_{L^{2}} \lesssim s^{1/4}\| xf_{\le s^{1/4}}\|_{L^{2}}\lesssim s^{1/4}s^{\delta},\label{6.33}\\
	&\| f_{\le s^{1/4}}\|_{L^{1}}\lesssim\| 1\|_{L^{2}}\| f_{\le s^{1/4}}\|_{L^{2}}=\Big(\int_{0}^{s^{1/4}} 1dx\Big)^{1/2} \| f_{\le s^{1/4}}\|_{L^{2}}\lesssim s^{1/8}\| f_{\le s^{1/4}}\|_{L^{1}}\lesssim s^{1/8}s^{\delta}\label{6.34}.
\end{align}
Then, from H$\ddot{o}$lder{'}s  inequality and the priori estimate of $f$ \eqref{2.1}, we  get
\begin{align*}
	\|\eqref{6.24} \|_{L^{2}}+\|\eqref{6.25} \|_{L^{2}}
	&\lesssim\beta\int_{0}^{t}\big( s^{2\delta_{N}}\| x^{2}f_{\le s^{1/4}}\|_{L^{2}}+ ss^{3\delta_{N}}\| xf_{\le s^{1/4}}\|_{L^{2}}\big)\big\| e^{is \Delta} f_{\le s^{1/4}}\big\|_{L^{\infty}}ds\\
	&\lesssim\beta\int_{0}^{t} \big(s^{2\delta_{N}}s^{1/4}\| xf_{\le s^{1/4}}\|_{L^{2}}+ ss^{3\delta_{N}}\| xf_{\le s^{1/4}}\|_{L^{2}}\big)\frac{1}{s^{3/2}}\|  f_{\le s^{1/4}}\|_{L^{1}}ds\\
	&\lesssim \beta\int_{0}^{t} \big(s^{2\delta_{N}}s^{1/4}s^{\delta} + ss^{3\delta_{N}}s^{\delta}\big)\frac{1}{s^{3/2}}s^{\frac{1}{8}}s^{\delta} ds\|(\mathcal{E},N,M)\|_{X}^{2}\\
&\lesssim  t^{\frac{3}{4}} \|(\mathcal{E},N,M)\|_{X}^{2},
\end{align*}
where we could choose $3\delta_{N}+2\delta \le \frac{1}{8}$. To estimate \eqref{6.26}, we integrate by parts for $\eta$ by using \eqref{6.1},
\begin{align}
	&\int_{0}^{t}\int_{R^{3}}\frac{\beta}{2}|\xi|\frac{|\xi|}{(1+|\xi|^{2})^{1/2}} s(\nabla_{\xi}\gamma(\xi,\eta))^{2}e^{is\gamma(\xi,\eta)} \nabla_{\eta}\widehat {f_{\le s^{1/4}}}(\xi-\eta,s)\overline{\widehat {f_{\le s^{1/4}}}}(\eta,s)d\eta ds,\label{6.30}\\
	&\int_{0}^{t}\int_{R^{3}}\frac{\beta}{2}|\xi|\frac{|\xi|}{(1+|\xi|^{2})^{1/2}} s(\nabla_{\xi}\gamma(\xi,\eta))^{2}e^{is\gamma(\xi,\eta)}\widehat {f_{\le s^{1/4}}}(\xi-\eta,s) \nabla_{\eta}\overline{\widehat {f_{\le s^{1/4}}}}(\eta,s)d\eta ds.\label{6.31}
\end{align}
The treatments of \eqref{6.30} and \eqref{6.31} are very similar to  \eqref{6.25}, so we directly have
\begin{align*}
	\|\eqref{6.30}\|_{L^{2}}+\|\eqref{6.31}\|_{L^{2}}
	\lesssim  t^{1-3\alpha}\|(\mathcal{E},N,M)\|_{X}^{2}.
\end{align*}
For the estimates of   \eqref{6.27}--\eqref{6.29},  by  H$\ddot{o}$lder{'}s inequality, \eqref{2.2}, and \eqref{6.34}, we have
\begin{align*}
	&\|\eqref{6.27}\|_{L^{2}}+	\|\eqref{6.28}\|_{L^{2}}+\|\eqref{6.29}\|_{L^{2}}\\
&\lesssim\int_{0}^{t}\beta \big(s^{\delta_{N}} \| f_{\le s^{1/4}}\|_{L_{2}}\| e^{is \Delta} f_{\le s^{1/4}}\|_{L^{\infty}}
	+s^{2\delta_{N}}\| xf_{\le s^{1/4}}\|_{L^{2}}\big\| e^{is \Delta} f_{\le s^{1/4}}\big\|_{L^{\infty}}\\
	&\quad\quad+ s^{3\delta_{N}}s\| f_{\le s^{1/4}}\|_{L^{2}}\big\| e^{is \Delta} f_{\le s^{1/4}}\big\|_{L^{\infty}}\big)ds\\
	&\lesssim  \|(\mathcal{E},N,M)\|_{X}^{2}.
\end{align*}
\noindent {\bf For the estimate of} $\|P_{\le k}h_{2}\|_{L_{2}}\lesssim t^{(11/4-3\delta)k} $. \\
From Bernstein{'}s inequality \eqref{1.11},  H$\ddot{o}$lder{'}s inequality and the fact $\| f_{\ge s^{1/4}}\|_{L^{2}}\lesssim \frac{1}{s^{1/4}}s^{\delta}$, we obtain
\begin{align*}
	\|P_{\le k}h_{2}\|_{L^{2}}	
	&=\bigg\|P_{\le k}\int_{0}^{t}\beta|\wedge|^2\frac{(-\Delta)^{1/2}}{(I-\Delta)^{1/2}}e^{{\pm}is\beta(\Delta^{2}-\Delta)^{1/2}}\mathcal{F}^{-1}\big(\widehat {e^{is \Delta}f_{\ge s^{1/4}}}\ast \widehat {e^{-is \Delta}\overline{f}}\big) ds\bigg\|_{L^{2}}\nonumber\\	
	&\lesssim 2^{2k}	\beta\int_{0}^{t} \big\|P_{\le k} (e^{is\Delta} f_{\ge s^{1/4}}\cdot e^{-is\Delta}  f)\big\|_{L^{2}}ds\\
	&\lesssim 2^{2k}2^{(3/4-3\delta)k}\beta \int_{0}^{t} \big\| e^{is\Delta}f_{\ge s^{1/4}}\cdot e^{-is\Delta}  f\big\|_{L^{(3/4-\delta)^{-1}}} ds\\
	&\lesssim 2^{2k}2^{(3/4-3\delta)k}\beta \int_{0}^{t} \| e^{is\Delta}f_{\ge s^{1/4}}\|_{L_{2}}\|e^{-is\Delta}  f\|_{L^{(1/4-\delta)^{-1}}} ds\\
		&\lesssim 2^{(11/4-3\delta)k}\beta\int_{0}^{t} \| e^{is\Delta}f_{\ge s^{1/4}}\|_{L_{2}} \frac{1}{s^{3/4+3\delta }}\| f\|_{L^{(3/4+\delta)^{-1}}}ds\\	
	&\lesssim 2^{(11/4-3\delta)k}\beta\int_{0}^{t}s^{\delta} \frac{1}{s^{1/4}} \frac{1}{s^{3/4+3\delta }}s^{\delta} ds  \\			
	&\lesssim 2^{(11/4-3\delta)k}.
\end{align*}
From  the estimates of \eqref{6.24}--\eqref{6.29} and $\|P_{\le k}h_{2}\|_{L^{2}}\lesssim 2^{(11/4-3\delta)k}$,  the proof of lemma \ref{lem6.2} is completed.
\end{proof}

%
%
%
%
%
\section{Decay Estimate}

In this section we are going to prove the following:
\setcounter{section}{6}\setcounter{equation}{0}

\begin{proposition}\label{prop7.1}
	$t\|e^{it|\alpha\nabla|}G\|_{B_{\infty,1}^{0}}\lesssim  \|(\mathcal{E},N,M)\|_{X}^2 $.
\end{proposition}
The detailed process of the proof can be found in \cite{ZFJ}.
\begin{proposition}\label{prop7.2}
	$t\|e^{i\beta(\Delta^{2}-\Delta)^{1/2}t}H\|_{L^{\infty}}\lesssim \|(\mathcal{E},N,M)\|_{X}^2 $.
\end{proposition}
\noindent {\bf Proof of Proposition \ref{prop7.2}.}\\
 Let us start with an $L^{\infty}\rightarrow L^1$ estimate of $e^{i\beta(\Delta^{2}-\Delta)^{1/2}t}H$,  which can be found in lemma \ref{lem7.1}. Next, we focus on the estimate of $\|\wedge^{\frac{1}{2}}H\|_{L^{1}}$, we have
\begin{align}
	\|\wedge^{\frac{1}{2}}H\|_{L^{1}}\lesssim \beta\int_{0}^{t}\||\wedge|^{\frac{5}{2}}(-\Delta)^{1/2}(I-\Delta)^{-1/2}e^{{\pm}is\beta(\Delta^{2}-\Delta)^{1/2}}(e^{is \Delta}f \cdot e^{-is \Delta}\overline{f})\|_{L^{1}} ds.\label{7.1}
\end{align}
Since $\|f\|_{L^{1}}\lesssim \|xf\|_{L^{2}}^{1/2}\|x^{2}f\|_{L^{2}}^{1/2}$, so
\begin{align*}
	\eqref{7.1}\lesssim \beta\int_{0}^{t} &\big\|x|\wedge|^{\frac{5}{2}}(-\Delta)^{1/2}(I-\Delta)^{-1/2}e^{-is\beta(\Delta^{2}-\Delta)^{1/2}}e^{is \Delta}f 	e^{-is \Delta}\overline{f} \big\|_{L_{x}^{2}}^{1/2}\\
	&\big\|x^{2}|\wedge|^{\frac{5}{2}}(-\Delta)^{1/2}(I-\Delta)^{-1/2}e^{-is\beta(\Delta^{2}-\Delta)^{1/2}}e^{is \Delta}f 	e^{-is \Delta}\overline{f} \big\|_{L_{x}^{2}}^{1/2}ds.
\end{align*}
We define
\begin{align}
	&\big\|x  |\wedge|^{\frac{5}{2}}(-\Delta)^{1/2}(I-\Delta)^{-1/2}e^{{\pm}is\beta(\Delta^{2}-\Delta)^{1/2}}e^{is \Delta}f 	e^{-is \Delta}\overline{f} \big\|_{L_{x}^{2}},\label{7.2}\\
	&\big\|x^{2}|\wedge|^{\frac{5}{2}}(-\Delta)^{1/2}(I-\Delta)^{-1/2}e^{{\pm}is\beta(\Delta^{2}-\Delta)^{1/2}}e^{is \Delta}f 	e^{-is \Delta}\overline{f} \big\|_{L_{x}^{2}}.\label{7.3}
\end{align}
\noindent {\bf For the estimate of  \eqref{7.2}.}
From Plancharel{'}s inequality we have
\begin{align}
\widehat{\eqref{7.2}}=	&\int_{R^{3}}|\xi|(1+|\xi|^{2})^{-1/2}|\xi|^{\frac{5}{2}}e^{is\gamma(\xi,\eta)}\nabla_{\xi}\widehat f(\xi-\eta,s) \overline{\widehat  f}(\eta,s)d\eta,\label{7.4}\\
	+&\int_{R^{3}}|\xi|(1+|\xi|^{2})^{-1/2}|\xi|^{\frac{5}{2}}e^{is\gamma(\xi,\eta)}is\nabla_{\xi}\gamma(\xi,\eta)\widehat f(\xi-\eta,s) \overline{\widehat  f}(\eta,s)d\eta,\label{7.5} \\
	+&\int_{R^{3}}\frac{\frac{7}{2}|\xi|^{\frac{5}{2}}+\frac{5}{2}|\xi|^{\frac{9}{2}}}{(1+|\xi|^{2})^{3/2}}e^{is\gamma(\xi,\eta)}\widehat f(\xi-\eta,s) \overline{\widehat  f}(\eta,s)d\eta.\label{7.6}
\end{align}
Using one of the factors $|\xi|$ in \eqref{6.1}, integrating by parts for $\eta$   in \eqref{7.4} and \eqref{7.5}, we can  obtain as main contributions:
\begin{align} &\int_{R^{3}}\frac{1}{4}|\xi|(1+|\xi|^{2})^{-1/2}|\xi|^{\frac{1}{2}}\frac{1}{is}\nabla_{\eta}\gamma(\xi,\eta)e^{is\gamma(\xi,\eta)}\nabla_{\xi}\nabla_{\eta}\widehat f(\xi-\eta,s)\overline{\widehat  f}(\eta,s)d\eta,\label{7.7}\\
	&\int_{R^{3}}\frac{1}{4}|\xi|(1+|\xi|^{2})^{-1/2}|\xi|^{\frac{1}{2}}\frac{1}{is}\nabla_{\eta}\gamma(\xi,\eta)e^{is\gamma(\xi,\eta)}\nabla_{\xi}\widehat f(\xi-\eta,s) \nabla_{\eta}\overline{\widehat  f}(\eta,s)d\eta,\label{7.8}	\\ &\int_{R^{3}}\frac{1}{4}|\xi|(1+|\xi|^{2})^{-1/2}|\xi|^{\frac{1}{2}}\nabla_{\xi}\gamma(\xi,\eta)\nabla_{\eta}\gamma(\xi,\eta)e^{is\gamma(\xi,\eta)}\nabla_{\eta}\widehat f(\xi-\eta,s) \overline{\widehat  f}(\eta,s)d\eta,\label{7.9}\\ &\int_{R^{3}}\frac{1}{4}|\xi|(1+|\xi|^{2})^{-1/2}|\xi|^{\frac{1}{2}}\nabla_{\xi}\gamma(\xi,\eta)\nabla_{\eta}\gamma(\xi,\eta)e^{is\gamma(\xi,\eta)}\widehat f(\xi-\eta,s) \nabla_{\eta}\overline{\widehat  f}(\eta,s)d\eta.\label{7.10}
\end{align}
 \eqref{7.9} and  \eqref{7.10}  can be handled in
a similar fashion, so  we only consider  \eqref{7.9}. Integrating by parts for $\eta$ again  in \eqref{7.9}, we obtain
\begin{align} &\int_{R^{3}}\frac{1}{4}|\xi|(1+|\xi|^{2})^{-1/2}|\xi|^{\frac{1}{2}}\frac{1}{is}\nabla_{\xi}\gamma(\xi,\eta)e^{is\gamma(\xi,\eta)}\nabla_{\eta}^{2}\widehat f(\xi-\eta,s) \overline{\widehat  f}(\eta,s)d\eta,\label{7.11}\\
	&\int_{R^{3}}\frac{1}{4}|\xi|(1+|\xi|^{2})^{-1/2}|\xi|^{\frac{1}{2}}\frac{1}{is}\nabla_{\xi}\gamma(\xi,\eta)e^{is\gamma(\xi,\eta)}\nabla_{\eta}\widehat f(\xi-\eta,s) \nabla_{\eta}\overline{\widehat  f}(\eta,s)d\eta\label{7.12}.
\end{align}
\eqref{7.7}, \eqref{7.8}, \eqref{7.11}, \eqref{7.12} can obtain as main contributions:
\begin{align}
	&\int_{R^{3}}\frac{1}{s} m_{1}(\xi,\eta)e^{is\gamma(\xi,\eta)}\nabla_{\eta}^{2}\widehat f(\xi-\eta,s) \overline{\widehat  f}(\eta,s)d\eta,\label{7.13}\\
	&\int_{R^{3}}\frac{1}{s} m_{1}(\xi,\eta)e^{is\gamma(\xi,\eta)}\nabla_{\eta}\widehat f(\xi-\eta,s) \nabla_{\eta}\overline{\widehat  f}(\eta,s)d\eta,\label{7.14}
\end{align}
where $ m_{1}(\xi, \eta)$ denotes symbols with homogenous bounds of order $\frac{3}{2}$ for large frequencies and which can be bound by $s^{\frac{3}{2}\delta_{N}}$. For the estimate of  \eqref{7.13},  form H$\ddot{o}$lder{'}s inequality and  the priori eatimate of $f$ \eqref{2.1}, we  have
\begin{align*}
	\|\eqref{7.13}\|_{L^{2}}
	&\lesssim \frac{1}{s}s^{\frac{3}{2}\delta_{N}}\| x^{2}f\|_{L^{2}}\| e^{is \Delta}f\|_{L^{\infty}}\\
&\lesssim \frac{1}{s}s^{\frac{3}{2}\delta_{N}}s^{1-2\alpha-\delta}\frac{1}{s^{1+\alpha}}\|(\mathcal{E},N,M)\|_{X}^{2}\\
&\lesssim \frac{1}{s} \|(\mathcal{E},N,M)\|_{X}^{2},	
\end{align*}
where  $3\alpha \ge \frac{3}{2}\delta_{N}-\delta$. Applying  dispersive estimate \eqref{2.1} and $\|f\|_{L^{4/3}} \lesssim \|xf\|_{L^{2}}^{\frac{1}{4}}\|x^2f\|_{L^{2}}^{\frac{1}{4}}$, we have
\begin{align*}
	\|\eqref{7.14}\|_{L^{2}}
	&\lesssim  \frac{1}{s}s^{\frac{3}{2}\delta_{N}}\|e^{is \Delta}xf\|_{L^{4}}^{2}	
\lesssim  \frac{1}{s}s^{\frac{3}{2}\delta_{N}}\frac{1}{s^{3/2}}\|xf\|_{L^{4/3}}^{2}\\
	&\lesssim  \frac{1}{s}s^{\frac{3}{2}\delta_{N}}\frac{1}{s^{3/2}}\|x^{2}f\|_{L^{2}}^{3/2}\|xf\|_{L^{2}}^{1/2}\\
&\lesssim \frac{1}{s} \|(\mathcal{E},N,M)\|_{X}^{2},	
\end{align*}
where we could choose  $3\alpha \ge \frac{3}{2}\delta_{N}-\delta$.
Finally, form H$\ddot{o}$lder{'}s inequality and  the priori eatimate of $f$ \eqref{2.1}, we  get
\begin{align*}	\|\eqref{7.6}\|_{L^{2}}
	&\lesssim s^{\frac{3}{2}\delta_{N}}\| f\|_{L^{2}}\| e^{is \Delta}f\|_{L^{\infty}}
	\lesssim \frac{1}{s} \|(\mathcal{E},N,M)\|_{X}^{2},
\end{align*}
where $\alpha \ge \frac{3}{2}\delta_{N}+\delta$.\\
\noindent {\bf For the estimate of  \eqref{7.3}.} From Plancharel{'}s inequality we have
\begin{align}
\widehat{\eqref{7.3}}=	+&\int_{R^{3}}|\xi|(1+|\xi|^{2})^{-1/2}|\xi|^{\frac{5}{2}}e^{is\gamma(\xi,\eta)}\nabla_{\xi}^{2}\widehat f(\xi-\eta,s) \overline{\widehat  f}(\eta,s)d\eta,\label{7.15}\\
	+&\int_{R^{3}}|\xi|(1+|\xi|^{2})^{-1/2}|\xi|^{\frac{5}{2}} is\nabla_{\xi}\gamma(\xi,\eta)e^{is\gamma(\xi,\eta)} \nabla_{\xi}\widehat f(\xi-\eta,s) \overline{\widehat  f}(\eta,s)d\eta,\label{7.16}\\
	+&\int_{R^{3}}|\xi|(1+|\xi|^{2})^{-1/2}|\xi|^{\frac{5}{2}} (is\nabla_{\xi}\gamma(\xi,\eta))^{2}e^{is\gamma(\xi,\eta)}\widehat f(\xi-\eta,s) \overline{\widehat  f}(\eta,s)d\eta,\label{7.17}\\	+&\int_{R^{3}}\frac{\frac{35}{4}|\xi|^{\frac{3}{2}}+\frac{45}{4}|\xi|^{\frac{7}{2}}+\frac{15}{4}|\xi|^{\frac{11}{2}}}{(1+|\xi|^{2})^{5/2}}e^{is\gamma(\xi,\eta)} \widehat f(\xi-\eta,s) \overline{\widehat  f}(\eta,s)d\eta, \label{7.18}\\
	+&\int_{R^{3}}\frac{\frac{7}{2}|\xi|^{\frac{5}{2}}+\frac{5}{2}|\xi|^{\frac{9}{2}}}{(1+|\xi|^{2})^{3/2}} e^{is\gamma(\xi,\eta)}\nabla_{\xi}\widehat f(\xi-\eta,s) \overline{\widehat  f}(\eta,s)d\eta,\label{7.19}\\		
	 +&\int_{R^{3}}\frac{\frac{7}{2}|\xi|^{\frac{5}{2}}+\frac{5}{2}|\xi|^{\frac{9}{2}}}{(1+|\xi|^{2})^{3/2}}is\nabla_{\xi}\gamma(\xi,\eta) e^{is\gamma(\xi,\eta)}\widehat f(\xi-\eta,s) \overline{\widehat  f}(\eta,s)d\eta.\label{7.20}
\end{align}
From H$\ddot{o}$lder{'}s inequality and priori eatimate of $f$ \eqref{2.1} to estimate \eqref{7.15} and \eqref{7.16},  we obtain
\begin{align*}
	\|\eqref{7.15}\|_{L^{2}}+\|\eqref{7.16}\|_{L^{2}}
	&\lesssim  s^{\frac{5}{2}\delta_{N}}\| x^{2}f\|_{L^{2}}\| e^{is \Delta}f\|_{L^{\infty}}+ss^{\frac{7}{2}\delta_{N}}\| xf\|_{L_{2}}\| e^{is \Delta}f\|_{L^{\infty}}\\
	&\lesssim   \big(s^{\frac{5}{2}\delta_{N}}s^{1-2\alpha-\delta}\frac{1}{s^{1+\alpha}}+ss^{\frac{7}{2}\delta_{N}}s^{\delta}\frac{1}{s^{1+\alpha}}\big)\|(\mathcal{E},N,M)\|_{X}^{2}	\\
&\lesssim  \|(\mathcal{E},N,M)\|_{X}^{2}.	
\end{align*}
where $\frac{7}{2}\delta_{N}+\delta\le \alpha$.  Integrating by parts for $\eta$   in \eqref{7.17}, we get
\begin{align}
	&\int_{R^{3}}\frac{1}{4}|\xi|^{\frac{3}{2}}(1+|\xi|^{2})^{-1/2} (\nabla_{\xi}\gamma(\xi,\eta))^{2}\nabla_{\eta}\gamma(\xi,\eta) se^{is\gamma(\xi,\eta)}\nabla_{\eta}\widehat f(\xi-\eta,s) \overline{\widehat  f}(\eta,s)d\eta,\label{7.21}\\
	&\int_{R^{3}}\frac{1}{4}|\xi|^{\frac{3}{2}}(1+|\xi|^{2})^{-1/2} (\nabla_{\xi}\gamma(\xi,\eta))^{2}\nabla_{\eta}\gamma(\xi,\eta) se^{is\gamma(\xi,\eta)}\widehat f(\xi-\eta,s) \nabla_{\eta}\overline{\widehat  f}(\eta,s)d\eta.\label{7.22}
\end{align}
 \eqref{7.21} and \eqref{7.22} are similar to \eqref{7.16},
 \eqref{7.18} is similar to \eqref{7.6}, so we directly get
 \begin{align*}
	\|\eqref{7.18}\|_{L^{2}}+\|\eqref{7.21}\|_{L^{2}}+\|\eqref{7.22}\|_{L^{2}}
	\lesssim  t^{1-3\alpha}\|(\mathcal{E},N,M)\|_{X}^{2}.
\end{align*}
  Similarly,  we also get
\begin{align*}
	\|\eqref{7.19}\|_{L^{2}}+	\|\eqref{7.20}\|_{L^{2}}
	&\lesssim  s^{\frac{5}{2}\delta_{N}}\big(\| xf\|_{L_{2}}\| e^{is \Delta}f\|_{L^{\infty}}
	+s\| f\|_{L_{2}}\| e^{is \Delta}f\|_{L^{\infty}}\big)
\lesssim  \|(\mathcal{E},N,M)\|_{X}^{2}.	
\end{align*}
In short
\begin{align*}
	\|\wedge^{\frac{1}{2}}H\|_{L^{1}}\lesssim \int_{0}^{t}\beta(\frac{1}{s})^{\frac{1}{2}}ds\|(\mathcal{E},N,M)\|_{X}^{2}\lesssim \beta t^{\frac{1}{2}}\|(\mathcal{E},N,M)\|_{X}^{2}.
\end{align*}
Thus, \begin{align}\|e^{i\beta(\Delta^{2}-\Delta)^{1/2}t}H\|_{L^{\infty}}\lesssim t^{-\frac{3}{2}}\|\wedge^{\frac{1}{2}}H\|_{L^{1}}\lesssim \beta t^{-\frac{3}{2}}t^{\frac{1}{2}}\|(\mathcal{E},N,M)\|_{X}\lesssim  \frac{1}{t}\|(\mathcal{E},N,M)\|_{X}.\label{7.23}
\end{align}

From  the estimates of \eqref{7.4}--\eqref{7.22},  the proof of Proposition \ref{prop7.2} is completed.
$\hfill\Box$\\

%
%
%
%
%

\section{Weighted Estimates for the Schr$\ddot{o}$dinger Component}

The purpose of this section is to prove:
\begin{proposition}\label{prop8.1}
	 $t^{-\delta}\|xF\|_{L^{2}}+t^{-1+2\alpha+\delta}\|x^{2}F\|_{L^{2}}\lesssim \|(\mathcal{E},N,M)\|_{X}^2$.
\end{proposition}
\noindent {\bf Proof of Proposition \ref{prop8.1}.}\\
To prove this Proposition \ref{prop8.1}, we first prove the following lemma.
\begin{lemma}\label{lem8.1}
Let space X be defined in \eqref{4.21}, we have
 \begin{align*}	
t^{-\delta}\|xF_{1}\|_{L^{2}}+t^{-1+2\alpha+\delta}\|x^{2}F_{1}\|_{L^{2}}\lesssim \|(\mathcal{E},N,M)\|_{X}^{2}.
\end{align*}	
\end{lemma}
The detailed process of the proof can be found in \cite{ZFJ}.\\

Next, we need to handle the weighted estimates of $F_{2}$. A key identity that we are going to use is
\begin{align}	&\nabla_{\xi}\rho=2\eta=-2\frac{2\sqrt{\eta^{2}+1}}{\beta\eta+\sqrt{\eta^{2}+1}}\nabla_{\eta}\rho+\frac{2\sqrt{\eta^{2}+1}}{\beta\eta+\sqrt{\eta^{2}+1}}\frac{\rho}{|\eta|}.\label{8.1}
\end{align}
\noindent {\bf For the estimate of} $\|xF_{2}\|_{L^{2}}\lesssim t^{\delta}\|(\mathcal{E},N,M)\|_{X}^{2}$. \\Applying $\nabla_{\xi}$ to $\widehat F_{2}$ gives the terms:
\begin{align}
	\nabla_{\xi}\widehat F_{2}&=\int_{0}^{t} \int_{R^{3}} e^{is\rho(\xi,\eta)}\nabla_{\xi}\widehat f(\xi-\eta,s)\widehat  h(\eta,s)d\eta ds,\label{8.2}\\
	&+\int_{0}^{t}\int_{R^{3}}is\nabla_{\xi}\rho(\xi,\eta) e^{is\rho(\xi,\eta)}\widehat f(\xi-\eta,s)\widehat  h(\eta,s)d\eta ds.\label{8.3}	
\end{align}
For the estimate of \eqref{8.2}, from  H$\ddot{o}$lder{'}s inequality and \eqref{7.23}, we have
\begin{align*}
	\|\eqref{8.2}\|_{L^{2}}
	&\lesssim\int_{0}^{t}\|xf\|_{L^{2}}\big\| e^{\mp is\beta(\Delta^{2}-\Delta)^{1/2}}h\big\|_{L^{\infty}}ds\\
 &\lesssim\int_{0}^{t}s^{\delta}\frac{1}{s}ds\|(\mathcal{E},N,M)\|_{X}^{2}\\
&\lesssim t^{\delta}\|(\mathcal{E},N,M)\|_{X}^{2}.
\end{align*}
Using \eqref{8.1} to integrate by parts for $\eta$ and $s$ in  \eqref{8.3},  we get the following  terms:
\begin{align}	
	&\int_{0}^{t}\int_{R^{3}}\frac{4\sqrt{\eta^{2}+1}}{\beta\eta+\sqrt{\eta^{2}+1}}e^{is\rho(\xi,\eta)}\nabla_{\eta}\widehat f(\xi-\eta,s)\widehat  h(\eta,s)d\eta ds,\label{8.4}\\
	&\int_{0}^{t}\int_{R^{3}} \frac{4\sqrt{\eta^{2}+1}}{\beta\eta+\sqrt{\eta^{2}+1}}e^{is\rho(\xi,\eta)}\widehat f(\xi-\eta,s)\nabla_{\eta}\widehat  h(\eta,s)d\eta ds\label{8.5},\\	
	&\int_{R^{3}}t\frac{2\sqrt{\eta^{2}+1}}{\beta\eta+\sqrt{\eta^{2}+1}}\frac{1}{|\eta|}e^{it\rho(\xi,\eta)}\widehat f(\xi-\eta,t)\widehat  h(\eta,t) d\eta,\label{8.6}\\
	&\int_{0}^{t}\int_{R^{3}}s\frac{2\sqrt{\eta^{2}+1}}{\beta\eta+\sqrt{\eta^{2}+1}}\frac{1}{|\eta|}e^{is\rho(\xi,\eta)}\partial_{s}\widehat f(\xi-\eta,s)\widehat  h(\eta,s)d\eta ds,\label{8.7}\\	
	&\int_{0}^{t}\int_{R^{3}}s\frac{2\sqrt{\eta^{2}+1}}{\beta\eta+\sqrt{\eta^{2}+1}}\frac{1}{|\eta|}e^{is\rho(\xi,\eta)}\widehat f(\xi-\eta,s)\partial_{s}\widehat  h(\eta,s)d\eta ds.\label{8.8}	
\end{align}
The estimate of \eqref{8.4} is almost identical to \eqref{8.2}, it can be derived in an identical fashion,  so we directly get
 \begin{align*}
	\|\eqref{8.4}\|_{L^{2}}
	\lesssim  t^{\delta}\|(\mathcal{E},N,M)\|_{X}^{2}.
\end{align*}
For the estimate of  \eqref{8.5},
recall that
\begin{align}
	h(t,x)=h(0,x)+H(t,x)=B_{0}(x)+i\beta^{-1}(\Delta^{2}-\Delta)^{-1/2}B_{1}(x)	+H(t,x),
\end{align}
and therefore
\begin{align}
	xh(t,x)=xB_{0}(x)+i\beta^{-1}(\Delta^{2}-\Delta)^{-\frac{1}{2}}xB_{1}(x)	-i\beta^{-1}\frac{1-2\Delta}{\Delta(1-\Delta)^\frac{3}{2}}B_{1}(x)+xH(t,x).
\end{align}
Finally,  from Sobolev{'}s embedding theorem, H$\ddot{o}$lder{'}s inequality and initial datas \eqref{1.7}, we have
\begin{align*}
	\|\eqref{8.5}\|_{L^{2}}
	&\lesssim \int_{0}^{t}
	\big\|e^{is \Delta}f\big\|_{L^{\infty}}\big\| e^{\mp is\beta(\Delta^{2}-\Delta)^{1/2}}xh\big\|_{L^{2}}ds\nonumber\\
				&\lesssim\int_{0}^{t}\big\|e^{is \Delta}f\big\|_{L^{\infty}}\Big(\| xB_{0}(x)\|_{L^{2}}+\| xB_{1}(x)	\|_{H^{1}}+\| \beta^{-1} B_{1}(x)\|_{H^{1}}+\| xH(t,x)\|_{L^{2}}\Big)ds\\
				&\lesssim \int_{0}^{t}\frac{1}{{\left\langle s \right\rangle}^{1+\alpha}}ds\|(\mathcal{E},N,M)\|_{X}^{2}\\
				&\lesssim t^{\delta}\|(\mathcal{E},N,M)\|_{X}^{2}.
\end{align*}
For the estimate of \eqref{8.6}, by   H$\ddot{o}$lder{'}s inequality and  Hardy{'}s inequality, we  get
\begin{align*}
	\|\eqref{8.6}\|_{L^{2}}
	\lesssim t\big\|e^{it \Delta}f\big\|_{L^{\infty}}\big\| e^{\mp it\beta(\Delta^{2}-\Delta)^{1/2}}\wedge^{-1}h\big\|_{L^{2}}	
	\lesssim \|(\mathcal{E},N,M)\|_{X}^{2}.
\end{align*}
For the estimates of   \eqref{8.7} and \eqref{8.8}, by  $ \beta(-\Delta)^{1/2}(I-\Delta)^{-1/2}\nabla\times(\nabla\times(\mathcal{E}\times
\overline {\mathcal{E}}))=e^{i\beta(\Delta^{2}-\Delta)^{1/2}t}\partial_{s}h$, $\mathcal{E}N+\mathcal{E}M= e^{it \Delta}\partial_{s}f$ and Sobolev{'}s embedding theorem, we have
\begin{align*}
	&\|\eqref{8.7}\|_{L^{2}}+	\|\eqref{8.8}\|_{L^{2}}\\
		&\lesssim \int_{0}^{t}s\big\|e^{is \Delta}\partial_{s}f\big\|_{L^{\infty}}\big\| e^{\mp i\beta(\Delta^{2}-\Delta)^{1/2}s}\wedge^{-1}h\big\|_{L^{2}}+s\big\|e^{is \Delta}f\big\|_{L^{\infty}}\big\| e^{\mp is\beta(\Delta^{2}-\Delta)^{1/2}}\wedge^{-1}\partial_{s}h\big\|_{L^{2}}ds\\
	&\lesssim \int_{0}^{t}s\|\mathcal{E}\|_{L^{\infty}}\|N\|_{L^{\infty}}+s\|\mathcal{E}\|_{L^{\infty}}\|M\|_{L^{\infty}}+\beta s\|\mathcal{E}\|_{L^{\infty}}\| \mathcal{E}\|_{H^{N+1}}\|\mathcal{E}\|_{L^{\infty}}ds\\
	&\lesssim t^{\delta}\|(\mathcal{E},N,M)\|_{X}^{3}.
\end{align*}
\noindent {\bf For the estimate of} $\|x^{2}F_{2}\|_{L^{2}}\lesssim t^{1-2\alpha-\delta}\|(\mathcal{E},N,M)\|_{X}^{2}$. \\Applying $\nabla_{\xi}^{2}$ to $\widehat F_{2}$ gives the terms:
\begin{align}
	\nabla_{\xi}^{2}\widehat F_{2}&=\int_{0}^{t} \int_{R^{3}} e^{is\rho(\xi,\eta)}\nabla_{\xi}^{2}\widehat f(\xi-\eta,s)\widehat  h(\eta,s)d\eta ds,\label{8.11}\\
	&+\int_{0}^{t}\int_{R^{3}}is\nabla_{\xi}\rho(\xi,\eta) e^{is\rho(\xi,\eta)}\nabla_{\xi}\widehat f(\xi-\eta,s)\widehat  h(\eta,s)d\eta ds,\label{8.12}\\
	&+\int_{0}^{t}\int_{R^{3}}(is\nabla_{\xi}\rho(\xi,\eta) )^{2}e^{is\rho(\xi,\eta)}\widehat f(\xi-\eta,s)\widehat  h(\eta,s)d\eta ds.\label{8.13}		
\end{align}
\noindent {\bf For the estimate of  \eqref{8.11}.}  From an $L^2 \times L^\infty$ application of H$\ddot{o}$lder{'}s inequality and \eqref{7.23}, we  get
\begin{align*}
	\|\eqref{8.11}\|_{L^{2}}
	&\lesssim\int_{0}^{t}\|x^2f\|_{L^{2}}\big\| e^{\pm is\beta(\Delta^{2}-\Delta)^{1/2}}h\big\|_{L^{\infty}}ds\\
&\lesssim \int_{0}^{t}s^{1-2\alpha-\delta}\frac{1}{s}ds\|(\mathcal{E},N,M)\|_{X}^{2}\\
	&\lesssim t^{1-2\alpha-\delta}\|(\mathcal{E},N,M)\|_{X}^{2}.
\end{align*}
\noindent {\bf For the estimate of  \eqref{8.12}.} We use \eqref{8.1} to integrate by parts for $\eta$ and $s$ in \eqref{8.12}, we have
\begin{align}
		&\int_{0}^{t}\int_{R^{3}} \frac{4\sqrt{\eta^{2}+1}}{\beta\eta+\sqrt{\eta^{2}+1}}e^{is\rho(\xi,\eta)}\nabla_{\eta}\nabla_{\xi}\widehat f(\xi-\eta,s)\widehat  h(\eta,s)d\eta ds,\label{8.14}\\
	&\int_{0}^{t}\int_{R^{3}} \frac{4\sqrt{\eta^{2}+1}}{\beta\eta+\sqrt{\eta^{2}+1}}e^{is\rho(\xi,\eta)}\nabla_{\xi}\widehat f(\xi-\eta,s)\nabla_{\eta}\widehat  h(\eta,s)d\eta ds,\label{8.15}\\
	&\int_{R^{3}}t\frac{2\sqrt{\eta^{2}+1}}{\beta\eta+\sqrt{\eta^{2}+1}}\frac{1}{|\eta|}e^{it\rho(\xi,\eta)}\nabla_{\xi}\widehat f(\xi-\eta,t)\widehat  h(\eta,t)d\eta, \label{8.16}\\	
	&\int_{0}^{t}\int_{R^{3}}s\frac{2\sqrt{\eta^{2}+1}}{\beta\eta+\sqrt{\eta^{2}+1}}\frac{1}{|\eta|}e^{is\rho(\xi,\eta)}\nabla_{\xi}\widehat f(\xi-\eta,s)\partial_{s}\widehat  h(\eta,s)d\eta ds,\label{8.17}\\
	&\int_{0}^{t}\int_{R^{3}}s\frac{2\sqrt{\eta^{2}+1}}{\beta\eta+\sqrt{\eta^{2}+1}}\frac{1}{|\eta|}e^{is\rho(\xi,\eta)}\partial_{s}\nabla_{\xi}\widehat f(\xi-\eta,s)\widehat  h(\eta,s)d\eta ds.\label{8.18}	
\end{align}
For the estimate of \eqref{8.14}, it can be derived in a straightforward fashion which is identical to  \eqref{8.11}, so we directly get
 \begin{align*}
	\|\eqref{8.14}\|_{L^{2}}
	\lesssim  t^{1-2\alpha-\delta}\|(\mathcal{E},N,M)\|_{X}^{2}.
\end{align*}
Applying  Hardy{'}s inequality, Sobolev{'}s embedding theorem and initial datas \eqref{1.7} in \eqref{8.15}, we have
\begin{align*}
		\|\eqref{8.15}\|_{L^{2}}
	&\lesssim \int_{0}^{t}\big\|e^{is \Delta}xf\big\|_{L^{6}}\big\| e^{\mp is\beta(\Delta^{2}-\Delta)^{1/2}}xh\big\|_{L^{3}}ds\\
	&\lesssim \int_{0}^{t}\frac{1}{s}\|x^{2}f\|_{L^{2}}\Big(\big\| xB_{0}(x)+xH(t,x)\big\|_{L^{3}}+\| xB_{1}(x)\|_{H^{1}}+\|  B_{1}(x)\|_{H^{N-3}}\Big)ds\\	
&\lesssim t^{1-2\alpha-\delta}\|(\mathcal{E},N,M)\|_{X}^{2}.
\end{align*}
The term \eqref{8.16} can be treated similarly to \eqref{8.15}, since $\widehat h/|\eta| $ plays the same role as $\nabla_{\eta}\widehat h$, and the factor of $t$ plays the same role of the integral in time, so we have
 \begin{align*}
	\|\eqref{8.16}\|_{L^{2}}
	\lesssim  t^{1-2\alpha-\delta}\|(\mathcal{E},N,M)\|_{X}^{2}.
\end{align*}
For the estimate of  \eqref{8.17}, from the fact that
$\|\wedge^{-1}e^{i\beta(\Delta^{2}-\Delta)^{1/2}s}\partial_{s}h\|_{L^{3}}\lesssim \beta s^{\delta}\frac{1}{s^{1+\alpha}}$, we have
\begin{align*}
	\|\eqref{8.17}\|_{L^{2}}
	\lesssim \int_{0}^{t}s\big\|e^{is \Delta}xf\big\|_{L^{6}}\big\| e^{\mp is\beta (\Delta^{2}-\Delta)^{1/2}}\wedge^{-1}\partial_{s}h\big\|_{L^{3}}ds
	\lesssim  t^{1-2\alpha-\delta}\|(\mathcal{E},N,M)\|_{X}^{2},
\end{align*}
where $\delta \le\alpha$.  For the estimate of   \eqref{8.18}, first, we consider, because
\begin{align*}
	e^{is \Delta}\partial_{s}xf&=\mathcal{F}^{-1}\bigg(\int_{R^{3}}s\nabla_{\xi}\phi(\xi,\eta)\widehat {\mathcal{E}}(\xi-\eta,s)\widehat  N(\eta,s) d\eta\bigg) +e^{is \Delta}xfN,\\
	&+\mathcal{F}^{-1}\bigg(\int_{R^{3}}s\nabla_{\xi}\rho(\xi,\eta)\widehat {\mathcal{E}}(\xi-\eta,s)\widehat  M(\eta,s) d\eta\bigg)+e^{is \Delta}xfM,
\end{align*}
so, we have
$\|e^{is \Delta}\partial_{s}xf\|_{L^{6}}\lesssim \frac{1}{s}s^{\delta}s^{\delta_{N}}.	$
Next, follows from  initial datas \eqref{1.7}, Hardy--Littlewood--Sobolev \eqref{lem2.3}  and   the linear dispersive estimate \eqref{2.5}, we have
\begin{align*}
	&\big\|	e^{is\beta(\Delta^{2}-\Delta)^{1/2}}\wedge^{-1}B_{0}(x)\big\|_{L^{3}}\lesssim\frac{1}{s^{\frac{1}{3}}}\big\|\wedge^{-\frac{2}{3}}B_{0}(x)\big\|_{L^{3/2}}
	 \lesssim\frac{1}{s^{\frac{1}{3}}}\|xB_{0}(x)\|_{L^{2}}\lesssim\frac{1}{s^{\frac{1}{3}}},\\
	&\big\|	e^{i\beta(\Delta^{2}-\Delta)^{1/2}s}\wedge^{-1}i\beta^{-1}(\Delta^{2}-\Delta)^{-1/2}B_{1}(x)\big\|_{L^{3}}
\lesssim\frac{1}{s^{\frac{1}{3}}}\big\|\wedge^{-\frac{5}{3}}\beta^{-1}B_{1}(x)\big\|_{L^{1}}\lesssim\frac{1}{s^{\frac{1}{3}}}.
\end{align*}
Moreover,  still with the help of  Product law and   the linear dispersive estimate \eqref{2.5}, we see that
\begin{align*}
	\big\|	e^{is\beta(\Delta^{2}-\Delta)^{1/2}}\wedge^{-1}H(t,x)\big\|_{L^{3}}
	&\lesssim \int_{0}^{t}\frac{1}{(s-r)^{\frac{1}{3}}}\big\|	|\wedge|^{\frac{4}{3}}(\mathcal{E}\times
	\overline {\mathcal{E}})\big\|_{L^{3/2}}dr\\
		&\lesssim \int_{0}^{t}\frac{1}{(s-r)^{\frac{1}{3}}}\|	\mathcal{E}\|_{L^{3}}\|	\mathcal{E}\|_{W_{2,3}}dr\\
				&\lesssim \int_{0}^{t}\frac{1}{(s-r)^{\frac{1}{3}}}s^{2\delta_{N}}\frac{1}{\sqrt r}r^{\frac{\delta}{2}}\frac{1}{\sqrt r}r^{\frac{\delta}{2}}dr\\
					&\lesssim	\frac{1}{s^{\frac{1}{3}}}s^{2\delta_{N}}	s^{\delta}.
\end{align*}
Finally, collect the estimates above, we have
\begin{align*}
	\|	\eqref{8.18}\|_{L^{2}}
	\lesssim \int_{0}^{t}s\big\|e^{is \Delta}\partial_{s}xf\big\|_{L^{6}}\big\| e^{\mp is\beta (\Delta^{2}-\Delta)^{1/2}}\wedge^{-1}h\big\|_{L^{3}}ds
	\lesssim t^{1-2\alpha-\delta}\|(\mathcal{E},N,M)\|_{X}^{2},
\end{align*}
where $3\delta_{N}\le \delta$.\\
\noindent {\bf For the estimate of  \eqref{8.13}.}
We split
$B( f, h)$ into  $B_{0}$,  $B_{1}$ and $ B_{2}$ with
\begin{align}
	&B_{0}(f,h)(t,\xi):=B(f,h_0)=\int_{0}^{t}\int_{R^{3}}(is\nabla_{\xi}\rho(\xi,\eta) )^{2}e^{is\rho(\xi,\eta)}\widehat f(\xi-\eta,s)\widehat  h_0(\eta,s)d\eta ds,	\label{8.19}\\
	&B_{1}(f,h)(t,\xi):=B(f,h_1)=\int_{0}^{t}\int_{R^{3}}(is\nabla_{\xi}\rho(\xi,\eta) )^{2}e^{is\rho(\xi,\eta)}\widehat f(\xi-\eta,s)\widehat  h_1(\eta,s)d\eta ds,	\label{8.20}\\
	&B_{2}(f,h)(t,\xi):=B(f,h_2)=\int_{0}^{t}\int_{R^{3}}(is\nabla_{\xi}\rho(\xi,\eta) )^{2}e^{is\rho(\xi,\eta)}\widehat f(\xi-\eta,s)\widehat  h_2(\eta,s)d\eta ds.	\label{8.21}
\end{align}
We then split $h$ as $h = h_0 + h_1 + h_2$, where $h_0$ is the
initial data and $h_1$ and $h_2$ are as in \eqref{6.22} and \eqref{6.23} respectively.\\
\noindent {\bf For the estimate  of $B_{0}$ in $L^{2}$.}

 From \eqref{8.1}, we have $\nabla_{\xi}\rho(\xi,\eta)=|\eta| $ and $\eta^{2}\sim \frac{2\sqrt{\eta^{2}+1}}{\beta\eta+\sqrt{\eta^{2}+1}}\rho -\frac{4\sqrt{\eta^{2}+1}}{\beta\eta+\sqrt{\eta^{2}+1}}\eta \nabla_{\eta}\rho$, then using this idenity to integrate by parts for $\eta$ and $s$ in  \eqref{8.19},
\begin{align}	
	&\int_{0}^{t}\int_{R^{3}}s\frac{4\sqrt{\eta^{2}+1}}{\beta\eta+\sqrt{\eta^{2}+1}}
	\eta e^{is\rho(\xi,\eta)}\nabla_{\eta}\widehat f(\xi-\eta,s)\widehat   h_{0}(\eta,s)d\eta ds,\label{8.22}\\
	&\int_{0}^{t}\int_{R^{3}}s\frac{4\sqrt{\eta^{2}+1}}{\beta\eta+\sqrt{\eta^{2}+1}}\eta e^{is\rho(\xi,\eta)}\widehat f(\xi-\eta,s)\nabla_{\eta}\widehat  h_{0}(\eta,s)d\eta ds,\label{8.23}	\\	
	&\int_{R^{3}}t^{2}\frac{2\sqrt{\eta^{2}+1}}{\beta\eta+\sqrt{\eta^{2}+1}}e^{it\rho(\xi,\eta)}\widehat f(\xi-\eta,t)\widehat  h_{0}(\eta,t)d\eta,\label{8.24}\\	
	&\int_{0}^{t}\int_{R^{3}}s^{2}\frac{2\sqrt{\eta^{2}+1}}{\beta\eta+\sqrt{\eta^{2}+1}}e^{is\rho(\xi,\eta)}\partial_{s}\widehat f(\xi-\eta,s)\widehat   h_{0}(\eta,s)d\eta ds.\label{8.25}	
\end{align}
 \eqref{8.22} is similar to \eqref{8.12}, so we have
  \begin{align*}
	\|\eqref{8.22}\|_{L^{2}}
	\lesssim  t^{1-2\alpha-\delta}\|(\mathcal{E},N,M)\|_{X}^{2}.
\end{align*}
For the estimate of \eqref{8.23}, first, we notice that
\begin{align*}
\eta\nabla_{\eta}\widehat  h_{0}=\eta\nabla_{\eta}\widehat  B_{0}+i\beta^{-1}\frac{\nabla_{\eta}\widehat  B_{1}}{(1+\eta^{2})^{1/2}}-i\beta^{-1}\frac{(2\eta^{2}+1)\widehat  B_{1}}{|\eta|(1+\eta^{2})^{3/2}}.
\end{align*}
 Then, from the linear dispersive estimate \eqref{2.5}, we have
\begin{align*}
\big\|e^{is\beta(\Delta^{2}-\Delta)^{1/2}}	\mathcal{F}^{-1}\eta\nabla_{\eta}\widehat  H_{0}\big\|_{L^{3}}\lesssim\frac{1}{s^{\frac{1}{3}}}s^{\frac{4}{3}\delta_{N}}.
\end{align*}
 Final,  collect the estimates above, we obtain
\begin{align*}
	\|\eqref{8.23}\|_{L^{2}}	
	\lesssim \int_{0}^{t}s\|e^{is \Delta}f\|_{L^{6}}\big\| e^{\mp is\beta (\Delta^{2}-\Delta)^{1/2}}\mathcal{F}^{-1}(\eta\nabla_{\eta}\widehat  h_{0})\big\|_{L^{3}}ds
	\lesssim  t^{1-2\alpha-\delta}\|(\mathcal{E},N,M)\|_{X}^{2},
\end{align*}
where $\delta_{N}\le\frac{3}{2}\delta$.  Similarly, from H$\ddot{o}$lder{'}s inequality and the linear dispersive estimate \eqref{2.5}, we get
\begin{align*}
	\|\eqref{8.24}\|_{L^{2}}	
	&\lesssim t^{2}\big\|e^{is \Delta}f\big\|_{L^{6}}\big\| e^{\mp is\beta (\Delta^{2}-\Delta)^{1/2}}  h_{0}\big\|_{L^{3}}ds
	\lesssim t^{1-2\alpha-\delta}\|(\mathcal{E},N,M)\|_{X}^{2},
\end{align*}
where
$	\big\|	e^{it\beta(\Delta^{2}-\Delta)^{1/2}} h_{0}\big\|_{L^{3}}\lesssim\frac{1}{s^{\frac{1}{3}}}s^{\frac{1}{3}\delta_{N}}$.
Finally,  from the fact that  $\mathcal{E}N+\mathcal{E}M= e^{it \Delta}\partial_{s}f$, we have
\begin{align*}
	\|	\eqref{8.25}\|_{L^{2}}
	\lesssim \int_{0}^{t}s^2\big\|e^{is \Delta}\partial_{s}f\big\|_{L^{6}}\big\| e^{\mp is\beta (\Delta^{2}-\Delta)^{1/2}}  h_{0}\big\|_{L^{3}}ds
	\lesssim t^{1-2\alpha-\delta}\|(\mathcal{E},N,M)\|_{X}^{2}.
\end{align*}
\noindent {\bf For the estimate of $B_{1}$ \eqref{8.20}.} Using \eqref{8.1} to integrate by parts for $\eta$ and s in  \eqref{8.20},  we have
\begin{align}
	&\int_{0}^{t}\int_{R^{3}} s\frac{4\sqrt{\eta^{2}+1}}{\beta\eta+\sqrt{\eta^{2}+1}}\eta e^{is\rho(\xi,\eta)}\nabla_{\eta}\widehat f(\xi-\eta,s)\widehat  h_{1}(\eta,s)d\eta ds,\label{8.26}\\
	&\int_{0}^{t}\int_{R^{3}} s\frac{4\sqrt{\eta^{2}+1}}{\beta\eta+\sqrt{\eta^{2}+1}}\eta e^{is\rho(\xi,\eta)}\widehat f(\xi-\eta,s)\nabla_{\eta}\widehat  h_{1}(\eta,s)d\eta ds,\label{8.27}\\
	&\int_{R^{3}}t^{2} \frac{2\sqrt{\eta^{2}+1}}{\beta\eta+\sqrt{\eta^{2}+1}}e^{it\rho(\xi,\eta)}\widehat f(\xi-\eta,t)\widehat  h_{1}(\eta,t)d\eta,\label{8.28}\\	
	&\int_{0}^{t}\int_{R^{3}}s^{2}\frac{2\sqrt{\eta^{2}+1}}{\beta\eta+\sqrt{\eta^{2}+1}} e^{is\rho(\xi,\eta)}\partial_{s}\widehat f(\xi-\eta,s)\widehat  h_{1}(\eta,s)d\eta ds,\label{8.29}\\	
	&\int_{0}^{t}\int_{R^{3}}s^{2}\frac{2\sqrt{\eta^{2}+1}}{\beta\eta+\sqrt{\eta^{2}+1}} e^{is\rho(\xi,\eta)}\widehat f(\xi-\eta,s)\partial_{s}\widehat h_{1}(\eta,s)d\eta ds.	\label{8.30}
\end{align}
 \eqref{8.26}, \eqref{8.28} and \eqref{8.29} are respectively similar to  \eqref{8.12}, \eqref{8.24} and \eqref{8.25}, and can be handled in a similar fashion, so we have
 \begin{align*}
	\|\eqref{8.26}\|_{L^{2}}+\|\eqref{8.28}\|_{L^{2}}+\|\eqref{8.29}\|_{L^{2}}
	\lesssim  t^{1-2\alpha-\delta}\|(\mathcal{E},N,M)\|_{X}^{2}.
\end{align*}
For the estimate of \eqref{8.27}, we use \eqref{8.1} to integrate   by parts for $\eta$ and $s$ in  \eqref{8.27},
\begin{align}
	&\int_{0}^{t}\int_{R^{3}}(\frac{4\sqrt{\eta^{2}+1}}{\beta\eta+\sqrt{\eta^{2}+1}})^2 e^{is\rho(\xi,\eta)}\nabla_{\eta}\widehat f(\xi-\eta,s)\nabla_{\eta}\widehat  h_{1}(\eta,s)d\eta ds,\label{8.31}\\
	&\int_{0}^{t}\int_{R^{3}}(\frac{4\sqrt{\eta^{2}+1}}{\beta\eta+\sqrt{\eta^{2}+1}})^2 e^{is\rho(\xi,\eta)}\widehat f(\xi-\eta,s)\nabla_{\eta}^2\widehat  h_{1}(\eta,s)d\eta ds,\label{8.32}\\
	&\int_{R^{3}}6t(\frac{\sqrt{\eta^{2}+1}}{\beta\eta+\sqrt{\eta^{2}+1}})^2\frac{1}{|\eta|}e^{it\rho(\xi,\eta)}\widehat f(\xi-\eta,t)\nabla_{\eta}\widehat  h_{1}(\eta,s)d\eta, \label{8.33}\\	
	&\int_{0}^{t}\int_{R^{3}}6s(\frac{\sqrt{\eta^{2}+1}}{\beta\eta+\sqrt{\eta^{2}+1}})^2\frac{1}{|\eta|}e^{is\rho(\xi,\eta)}\widehat f(\xi-\eta,s)\partial_{s}\nabla_{\eta}\widehat  h_{1}(\eta,s)d\eta ds,\label{8.34}\\
	&\int_{0}^{t}\int_{R^{3}}6s(\frac{\sqrt{\eta^{2}+1}}{\beta\eta+\sqrt{\eta^{2}+1}})^2\frac{1}{|\eta|}e^{is\rho(\xi,\eta)}\partial_{s}\widehat f(\xi-\eta,s)\nabla_{\eta}\widehat  h_{1}(\eta,s)d\eta ds.\label{8.35}	
\end{align}
The terms obtained by doing so are of the
type \eqref{8.14}-\eqref{8.18} (or easier), except  the following  terms:  \eqref{8.32}, \eqref{8.34} and \eqref{8.35}.
For the estimate of  \eqref{8.32}, from \eqref{2.1} and Lemma \ref{lem6.2}, we have
\begin{align*}
	\|\eqref{8.32}	\|_{L^{2}}
	&\lesssim \int_{0}^{t}\big\|e^{is \Delta}f\big\|_{L^{\infty}}\big\| e^{\mp is\beta (\Delta^{2}-\Delta)^{1/2}}  x^2h_{1}\big\|_{L^{2}}ds
	\lesssim t^{1-2\alpha-\delta}\|(\mathcal{E},N,M)\|_{X}^{2},
\end{align*}
where $\alpha+3\delta\le \frac{1}{12}$.
For the estimate of \eqref{8.34}, first, due to
\begin{align*}
e^{\pm is\beta (\Delta^{2}-\Delta)^{1/2}}\partial_{s}xh_{1}&=\mathcal{F}^{-1}\bigg(\int_{R^{3}}|\xi|(1+|\xi|^{2})^{-1/2}|\xi|^{2}s\nabla_{\xi}\rho(\xi,\eta)\widehat {\mathcal{E}}(\xi-\eta,s)\widehat {\mathcal{E}}(\eta,s) d\eta\bigg) +e^{is \Delta}xf\mathcal{E},
\end{align*}
thus
\begin{align*}
\big\|e^{\pm is\beta (\Delta^{2}-\Delta)^{1/2}}\partial_{s}xh_{1}\big\|_{L^{6}}
	\lesssim \frac{1}{s}s^{\delta}s^{3\delta_{N}}.
\end{align*}
 Second, follows from  initial datas \eqref{1.7} and   the linear dispersive estimate \eqref{2.5}, we have
\begin{align*}
\big\|	e^{is\beta(\Delta^{2}-\Delta)^{1/2}}\wedge^{-1}(\mathcal{E}_{0}(x)+F(t,x))\big\|_{L^{6}}\lesssim\frac{1}{s^{\frac{1}{3}}}s^{\frac{2}{3}\delta_{N}}	s^{\delta}.
\end{align*}
 Final, through the above analysis, we have
\begin{align*}
	\|	\eqref{8.34}\|_{L^{2}}
	&\lesssim \int_{0}^{t}s\big\|e^{is \Delta}\wedge^{-1}f\big\|_{L^{6}}\big\| e^{\mp is\beta (\Delta^{2}-\Delta)^{1/2}}\partial_{s}xh_{1}\big\|_{L^{3}}ds
	\lesssim t^{1-2\alpha-\delta}\|(\mathcal{E},N,M)\|_{X}^{2},		
\end{align*}
where $\frac{11}{3}\delta_{N}\le \delta$.
Finally,   from the fact that $\mathcal{E}N+\mathcal{E}M= e^{it \Delta}\partial_{s}f$, we have
\begin{align*}
	\|	\eqref{8.35}\|_{L^{2}}
	&\lesssim \int_{0}^{t}s\big\|e^{is \Delta}\partial_{s}f\big\|_{L^{6}}\big\| e^{\mp is\beta (\Delta^{2}-\Delta)^{1/2}}  x^2h_{1}\big\|_{L^{3}}ds\\
	&\lesssim \int_{0}^{t}s\big(\|\mathcal{E}\|_{L^{\infty}}\|\|N\|_{L^{\infty}} +\|\mathcal{E}\|_{L^{\infty}}\|M\|_{L^{\infty}}\big)\big\|e^{\mp is\beta(\Delta^{2}-\Delta)^{1/2}} x^2h_{1}(\eta,s)\big\|_{L^{2}}ds\\
	&
	\lesssim t^{1-2\alpha-\delta}\|(\mathcal{E},N,M)\|_{X}^{2}.
\end{align*}
For the estimate of \eqref{8.30}, from H$\ddot{o}$lder{'}s inequality and Sobolev{'}s embedding theorem,  we obtain
\begin{align*}
	\|	\eqref{8.30}\|_{L^{2}}	
	\lesssim \int_{0}^{t}s^2\big\|e^{is \Delta}f\big\|_{L^{6}}\big\| e^{\mp is\beta (\Delta^{2}-\Delta)^{1/2}}  \partial_{s}h_{1}\big\|_{L^{3}}ds
	\lesssim  t^{1-2\alpha-\delta}\|(\mathcal{E},N,M)\|_{X}^{2}.
\end{align*}
\noindent {\bf For the estimate of $B_{1}$ \eqref{8.21}.}  To estimate $B_{2}$ we decompose it further according to  the frequency $\eta$. Let $X$ be a smooth positive radial and compactly supported function which equals 1 on $\left[0,1\right]$ and vanishes on $\left[2,+\infty\right)$, and define $X_{\le k} =X(\frac{\cdot}{K})$ and $X_{\ge K}=1-X_{\le K}$. Let $l$ be a positive number to be determined later, define
\begin{align}
	&B_{2}^{low}( f, h_{2})=\int_{0}^{t}\int_{R^{3}}(is\nabla_{\xi}\rho(\xi,\eta) )^{2}X_{\le s^{-l}}e^{is\rho(\xi,\eta)}\widehat f(\xi-\eta,s)\widehat  h_{2}(\eta,s)d\eta ds,\label{8.36}\\
	&B_{2}^{high}( f, h_{2})=\int_{0}^{t}\int_{R^{3}}(is\nabla_{\xi}\rho(\xi,\eta) )^{2}X_{\ge s^{-l}}e^{is\rho(\xi,\eta)}\widehat f(\xi-\eta,s)\widehat  h_{2}(\eta,s)d\eta ds.\label{8.37}
\end{align}	
\noindent {\bf For the estimate of  $B_{2}^{low}$.}  Applying H$\ddot{o}$lder{'}s  and Bernstein{'s} inequalities \eqref{1.11}, we obtain
\begin{align*}
	\|\eqref{8.36}\|_{L^{2}}
	&\lesssim \int_{0}^{t}(is)^{2}\frac{1}{s^{2l}}\Big\|\int_{R^{3}} \big|e^{-is |\xi-\eta|^2}\widehat f(\xi-\eta,s)P_{\le \log_2s^{-l}}e^{ \pm is\beta|\eta|\sqrt{{|\eta|}^{2}+1}}\widehat h_{2}(\eta,s)\big|d\eta\Big\|_{L^{2}}ds\\
	&\lesssim  \int_{0}^{t}(is)^{2}\frac{1}{s^{2l}}\big\|\mathcal{F}({|e^{is \Delta}f|})\ast \mathcal{F}({|P_{\le \log_2s^{-l}}e^{\mp is\beta (\Delta^{2}-\Delta)^{1/2}}  h_{2}}|)\big\|_{L^{2}}ds\\
	&\lesssim \int_{0}^{t}(is)^{2}\frac{1}{s^{2l}}\|e^{is \Delta}f\|_{L^{6}}\big\| P_{\le \log_2s^{-l}}e^{\mp is\beta (\Delta^{2}-\Delta)^{1/2}}h_{2}\big\|_{L^{3}}ds\\
		&\lesssim \int_{0}^{t}(is)^{2}\frac{1}{s^{l/2}}\frac{1}{s^{2l}}\|e^{is \Delta}f\|_{L^{6}}\big\| P_{\le \log_2s^{-l}}h_{2}\big\|_{L^{2}}ds\\
	&\lesssim \int_{0}^{t}s^{2}\frac{1}{s^{2l}}\frac{1}{s}s^{\delta}\frac{1}{s^{l/2}}2^{11/4-3\delta}ds\|(\mathcal{E},N,M)\|_{X}^{2}\\		
	&\lesssim t^2\frac{1}{t^{11/4}}t^{(1+3l)\delta}\|(\mathcal{E},N,M)\|_{X}^{2}	\\
	&\lesssim t^{1-2\alpha-\delta}\|(\mathcal{E},N,M)\|_{X}^{2}.
\end{align*}
And then we just need to satisfy
\begin{equation*}
	\left\{\!\!
	\begin{array}{lc}
		1+3l<2,
		&\\ 2-\frac{21}{4}l<\frac{3}{2}. &
	\end{array}
	\right.
\end{equation*}
That means $\frac{16}{63}<l<\frac{1}{3}$. We then choose $l=\frac{1}{3}-\frac{1}{30}$.\\
\noindent {\bf For the estimate of $B_{2}^{high}$ in $L_{2}$.}
To estimate the component $B_{2}^{high}$, we use \eqref{8.1}  to integrate by parts for time and frequency. By doing this we obtain again terms of the
form \eqref{8.14}-\eqref{8.18} (or easier ones), or the analogues of \eqref{8.24}, \eqref{8.25} with $h_{2}$ instead
of $h_{1}$, plus the following term:
\begin{align}
\int_{0}^{t}\int_{R^{3}} (\frac{4\sqrt{\eta^{2}+1}}{\beta\eta+\sqrt{\eta^{2}+1}})^2X_{\ge s^{-l}}e^{is\rho(\xi,\eta)}\widehat f(\xi-\eta,s)\nabla_{\eta}^{2}\widehat  h_{2}(\eta,s)d\eta ds.\label{8.38}
\end{align}
For the estimate of \eqref{8.38}, from H$\ddot{o}$lder{'}s inequality,  Sobolev{'}s embedding theorem and Bernstein{'}s inequality \eqref{1.11}, we get
\begin{align*}
	\|\eqref{8.38}\|_{L_{2}}
	&\lesssim \int_{0}^{t}\big\|e^{is \Delta}f\big\|_{L^{6}}\big\| P_{\ge \log_2s^{-l}}e^{\pm is\beta (\Delta^{2}-\Delta)^{1/2}}x^2h_{2}\big\|_{L^{3}}ds\\
	&\lesssim \int_{0}^{t}\big\|e^{is \Delta}f\big\|_{L_{6}}\big\| P_{\ge \log_2s^{-l}} \wedge^{\frac{1}{2}} x^{2}h_{2}\big\|_{L_{2}}ds\\
	&\lesssim \int_{0}^{t}\frac{1}{s}s^{\delta}s^{\frac{l}{2}}\|\wedge x^{2}h_{2}\|_{L_{2}}ds\\
	&\lesssim \int_{0}^{t}\frac{1}{s}s^{\delta}s^{\frac{l}{2}}s^{1-3\alpha}ds\|(\mathcal{E},N,M)\|_{X}^{2}\\	
	&\lesssim t^{\frac{l}{2}+\delta}t^{1-3\alpha}\|(\mathcal{E},N,M)\|_{X}^{2}\\
	&\lesssim t^{1-2\alpha-\delta}\|(\mathcal{E},N,M)\|_{X}^{2}.
\end{align*}
Similarly, we just need to satisfy $1-2\alpha-\delta\ge\frac{l}{2}+\delta+1-3\alpha$, such that
$\alpha\ge \frac{2}{l}+2\delta$, this inequality holds true provided $\delta\le \frac{2}{135}$. From  the estimates of \eqref{8.2}--\eqref{8.38} and lemma \ref{lem8.1},  the proof of Proposition \ref{prop8.1} is completed.$\hfill\Box$\\

This completes the proof of Theorem \ref{1.1}.

\section*{Acknowledgment} The  authors are supported in part by the
National Science Foundation of China(grant 11971166).


\begin{thebibliography}{99}
\addcontentsline{toc}{section}{References}
\bibitem{4} J.  Stamper, D. Tidman, \textit{Magnetic field generation due to radiation pressure in a laser-produced
plasma},  Phys. Fluids, 1973, 16(11): 2024-2025.


\bibitem{H} X. He, \textit{The pondermotive force and magnetic field generation effects resulting from the non-linear interaction between plasma-wave and particles (in Chinese)}, Acta. Phys. Sinica, 1983, 32: 325-337.



\bibitem{ZGG}  J. Zhang, C. Guo, B. Guo, \textit{On the Cauchy problem for the magnetic Zakharov system}, Monatshefte F\"ur Mathematik, 2013, 170(1): 89-111.


\bibitem{BJ}  B. Guo, J. Zhang,
\textit{Well-posedness of the Cauchy problem for the magnetic Zakharov type system}, Nonlinearity,
2011, 24: 2191-2210.




\bibitem{ZBD} Z. Gan, B. Guo, D. Huang, \textit{Blow-up and nonlinear instability for the magnetic Zakharov system},  Journal of Functional Analysis, 2013, 265(6): 953-982.



\bibitem{ZBLJ} Z. Gan, B. Guo, L. Han, J. Zhang, \textit{Virial type blow-up solutions for the Zakharov system with magnetic field in a cold plasma}, Journal of Functional Analysis, 2011, 261: 2508-2528.

\bibitem{ZYT3} Z. Gan, Y. Ma, T. Zhong, \textit{Some remarks on the blow-up rate for the 3D magnetic Zakharov system},  Journal of Functional Analysis, 2015, 269: 2505-2529.


\bibitem{HZGG}    L.  Han ,  J. Zhang , Z.  Gan,  B. Guo,   \textit{On the limit behavior of the magnetic Zakharov system},  Science China Mathematics, 2012, 55(3): 509-540.

\bibitem{ZHG} J. Zhang,  L. Han, B. Guo,  \textit{On the nonlinear Schr$\ddot{o}$dinger limit of the magnetic Zakharov system},  Nonlinear Analysis Theory Methods \& Applications, 2012, 75(10): 4090-4103.





\bibitem{HZGG2} L.  Han ,  J. Zhang , Z.  Gan,  B. Guo,
\textit{Cauchy problem for the Zakharov system arising from hot plasma with low regularity data}, Communications in mathematical sciences, 2012, 11(2): 403-420.

\bibitem{KHP} S. Kun, J. Huang, Z. Pan,  \textit{Sharp threshold of global existence for the generalized Zakharov system with three-dimensional magnetic field in the subsonic limit},  Mathematical notes, 2015, 97(3): 450-467.

\bibitem{WG} X. Wu, B. Guo,  \textit{The Well-posedness and Blow-up rate of Solution for the Generalized Zakharov equations with Magnetic field in $R^d$}, DOI:10.48550/arXiv.1412.4306, 2014.

  \bibitem{WG2} X. Wu, B. Guo,  \textit{Qualitative Analysis of Solutions for the Generalized Zakharov Equations with Magnetic Field in $R^d$}, Indiana University Mathematics Journal, 2021, 70(1): 49-79.

\bibitem{YZZ} Y. Yang, R. Zhou, S. Zhu, \textit{On the Singular Solutions to a Generalized Magnetic Zakharov Model}, Acta Mathematica Scientia (Series A), 2022, 42(1): 70-85.


\bibitem{JJ} J. Bourgain, J. Colliander, \textit{On well posedness of the Zakharov system}, International Mathematics Research Notices, 1996,  515-546.


\bibitem{JTG} J. Ginibre, Y. Tsutsumi, G. Velo, \textit{On the Cauchy Problem for the Zakharov System},  Journal
of Functional Analysis, 1997, 151(2): 384-436.




\bibitem{DHS} D. Fang, H. Pecher, S. Zhong, \textit{Low regularity global well-posedness for the two-dimensional Zakharov system},   Analysis, 2009, 29:  265-282.



\bibitem{M} F. Merle, \textit{Lower bounds for the blowup rate of solutions of the Zakharov equation in dimension two},   Communications on Pure and Applied Mathematics, 1996, 49(8): 765-794.

    \bibitem{MN2} N. Masmoudi, K. Nakanishi, \textit{Energy convergence for singular limits
of Zakharov type systems}, Inventiones mathematicae, 2008, 172(3): 535-583.


\bibitem{OT1} T. Ozawa,  Y. Tsutsumi,  \textit{The nonlinear Schr\"{o}dinger
limit and the initial layer of the Zakharov equations}, Differential Integral Equations, 1992, 5(4): 721-745.

\bibitem{OT2}T. Ozawa, Y. Tsutsumi, \textit{Existence and smoothing effect of solutions for the Zakharov equations}, Publications of the Research Institute for Mathematical Sciences, 2009, 28(3): 329-361.


\bibitem{ZK} Z. Guo, K. Nakanishi, \textit{ Small energy scattering for the Zakharov system with radial symmetry},  International Mathematics Research Notices, 2013(9): 9.



\bibitem{PNJ1} P. Germain,  N. Masmoudi, J. Shatah, \textit{Global solutions for 3D quadratic Schr$\ddot{o}$dinger equations},  International Mathematics Research Notices, 2009, 2009(3): 414-414.







\bibitem{PNJ2} P. Germain, N. Masmoudi, J. Shatah,   \textit{Global solutions for the gravity  water waves equation in dimension 3}, Comptes rendus-Math\;ematique, 2009, 347(15): 897-902.






\bibitem{PN} P. Germain, N. Masmoudi,   \textit{Global existence for the Euler-Maxwell system}, Mathematics, 2011, 47(3).

\bibitem{G} P. Germain,   \textit{Global existence for coupled Klein-Gordon equations with different speeds},  Annales-Institut Fourier, 2010, 61(6): 2463-2506.



\bibitem{PNJ3}  P. Germain, N. Masmoudi, J. Shatah,  \textit{Global solutions for 2D quadratic Schr$\ddot{o}$dinger equations}, Journal De Math\'{e}matiques Pures Et Appliqu\'{e}es, 2012, 97(5): 505-543.






    \bibitem{PNJ4}  P. Germain, N. Masmoudi, J. Shatah,  \textit{Global Existence for Capillary Water Waves},  Communications on Pure and Applied Mathematics, 015, 68(4): 625-687.



\bibitem{KP}  J. Kato,  F. Pusateri,  \textit{A new proof of long range scattering for
	critical nonlinear Schr\"{o}dinger equations},  Differential and Integral Equations, 2012, 24(9):  923-940.



\bibitem{ZFJ} Z. Hani, F. Pusateri, J. Shatah, \textit{ Scattering for the Zakharov System in 3 Dimensions},  Communications in Mathematical Physics, 2013, 322(3): 731-753.

    \bibitem{GAB} Y. Guo, A. Ionescu, B. Pausader,  \textit{Global solutions of the Euler-Maxwell two-fluid system in 3D},  Ann.
of Math, 2016, 183 (2), 377-498.

\bibitem{GBK2} Y. Guo, B. Pausader,  K. Widmayer, \textit{Global Axisymmetric Euler Flows with Rotation},   DOI:10.48550/arXiv.2109.01029, 2021.

\bibitem{GHZ}  Y. Guo, L. Han, J. Zhang, \textit{Absence of Shocks for One Dimensional Euler-Poisson System},
Archive for Rational Mechanics and Analysis, 2017, 223(3): 1057-1121.



\bibitem{ZLB} Z. Guo,   L. Peng,   B. Wang, \textit{Decay estimates for a class of wave equations},  Journal of Functional Analysis, 2008, 254(6): 1642-1660.



\bibitem{FJ}  F. Pusateri, J. Shatah,  \textit{Space-time resonances and the null condition for (first order) systems of wave	equations},
 Communications on Pure and Applied Mathematics, 2013, 66(10): 1495-1540.

































\end{thebibliography}
\end{document}